\documentclass[11pt, a4paper]{amsart}
\usepackage{amsmath,amssymb, amsthm} 
\usepackage{bbm}             
\usepackage[pdftex]{hyperref}  

\theoremstyle{plain}
    \newtheorem{thm}{Theorem}
    
    \newtheorem{prop}{Proposition}[section]
    \newtheorem{otherthm}[prop]{Theorem}
    \newtheorem{corol}[prop]{Corollary}
    \newtheorem*{scho}{Scholium}
    \newtheorem{lemma}[prop]{Lemma}
    \newtheorem{subl}[prop]{Sublemma}
    
	\newtheorem*{propdef}{Proposition/Definition}
\theoremstyle{definition}
    
\theoremstyle{remark}
    \newtheorem{rem}[prop]{Remark}
    
    \newtheorem*{conj}{Conjecture}
    \newtheorem*{claim}{Claim}
    \newtheorem*{ack}{Acknowledgements}

\numberwithin{equation}{section}

\def\bysame{\leavevmode\hbox to3em{\hrulefill}\thinspace} 
\newcommand{\R}{\mathbb{R}}\newcommand{\Z}{\mathbb{Z}}

\newcommand{\cB}{\mathcal{B}}\newcommand{\cC}{\mathcal{C}}\newcommand{\cD}{\mathcal{D}}

\newcommand{\cL}{\mathcal{L}}
\newcommand{\cM}{\mathcal{M}}\newcommand{\cN}{\mathcal{N}}\newcommand{\cO}{\mathcal{O}}\newcommand{\cP}{\mathcal{P}}
\newcommand{\cR}{\mathcal{R}}
\newcommand{\cU}{\mathcal{U}}\newcommand{\cV}{\mathcal{V}}
\newcommand{\cY}{\mathcal{Y}}\newcommand{\cZ}{\mathcal{Z}}
\newcommand{\eps}{\varepsilon}
\renewcommand{\epsilon}{\varepsilon}
\renewcommand{\setminus}{\smallsetminus}
\renewcommand{\emptyset}{\varnothing}

\newcommand{\id}{\mathrm{id}}

\newcommand{\Diff}{\mathrm{Diff}}

\DeclareMathOperator*{\supp}{supp}

\DeclareMathOperator{\cl}{cl}
\DeclareMathOperator{\ess}{ess}
\renewcommand{\phi}{\varphi}
\newcommand{\wed}{{\mathord{\wedge}}}
\newcommand{\charac}{{\mathbbm{1}}}
\newcommand{\bkappa}{{\boldsymbol{\kappa}}}
\newcommand{\blambda}{{\boldsymbol{\lambda}}}

\newcommand{\bnu}{{\boldsymbol{\nu}}}
\newcommand{\NUH}{\mathrm{Nuh}}      
\newcommand{\Var}{\mathrm{Var}}
\newcommand{\class}{\mathrm{Phc}}
\newcommand{\block}{\mathrm{Bl}}

\newcommand{\arxiv}[1]{\href{http://arxiv.org/abs/#1}{\tt arXiv:{#1}}}
\newcommand{\doi}[1]{\href{http://dx.doi.org/#1}{\tt doi:{#1}}}
\newcommand{\mr}[1]{\href{http://www.ams.org/mathscinet-getitem?mr=#1}{\tt MR{#1}}}

\begin{document}

\title[Nonuniform Hyperbolicity and Global Dominated Splittings]
{Nonuniform Hyperbolicity, Global Dominated Splittings and Generic Properties of Volume-Preserving Diffeomorphisms}

\author[A.~Avila]{Artur Avila}

\author[J.~Bochi]{Jairo Bochi}

\thanks{Both authors are partially supported by CNPq--Brazil. 
This research was partially
conducted during the period Avila served as a Clay Research Fellow.}

\begin{abstract}
We study generic volume-preserving diffeomorphisms on compact manifolds.
We show that the following property holds generically in the $C^1$ topology:
Either there is at least one zero Lyapunov exponent at almost every point,
or the set of points with only non-zero exponents forms an ergodic component.
Moreover, if this nonuniformly hyperbolic component has positive measure 
then it is essentially dense in the manifold
(that is, it has a positive measure intersection with any nonempty open set)
and 
there is a global dominated splitting.
For the proof we establish some new properties of independent interest
that hold $C^r$-generically for any $r \ge 1$,
namely: the continuity of the ergodic decomposition,
the persistence of invariant sets, and
the $L^1$-continuity of Lyapunov exponents.
\end{abstract}

\date{\today}

\maketitle




\section{Introduction}


Hyperbolicity is a fundamental concept in Differentiable Dynamical Systems. Its strongest form is \emph{uniform} hyperbolicity: it requires that the tangent bundle splits into uniformly contracting and expanding subbundles. Such dynamics is evidently ``chaotic'', that is, sensitive to the initial conditions. Moreover, these properties are robust under perturbations. Uniform hyperbolicity was studied by Smale, Anosov, Sinai and many others who obtained a profusion of consequences. 

Concurrent with the development of the uniformly hyperbolic theory, it became clear that it leaves out many chaotic dynamical systems of interest. This motivated the introduction of more flexible forms of hyperbolicity. 

In the presence of an invariant probability measure, Oseledets theorem guarantees the existence of Lyapunov exponents at almost every point. These numbers measure the asymptotic growth of tangent vectors under the dynamics. \emph{Nonuniform} hyperbolicity only requires that they are non-zero. As it was shown by Pesin and Katok, this condition allows for the development of a rich theory (invariant manifolds, periodic points etc). 
This theory has a strong measure-theoretic flavor: the Lyapunov exponents, the Oseledets subbundles and the invariant manifolds are only defined almost everywhere, and vary only measurably with the point.

Other relaxed versions of the notion of uniform hyperbolicity, initially developed having in mind the understanding of robust dynamical properties, were partial hyperbolicity and projective hyperbolicity (dominated splittings). While these keep some uniform requirements as the existence of continuous subbundles, neutral directions are also allowed. Those concepts later played an important role in the development of the theory of $C^1$-generic dynamics.

Much more information about these developments can be found in the books 
\cite{BP_book} and \cite{BDV_book}.
For an extensive current panorama of $C^1$-generic dynamics, see \cite{Cro}.



\medskip

In this paper we deal with 
conservative (i.e.\ preserving a smooth volume form) 
diffeomorphisms,
more precisely with $C^1$-generic ones.
One of our goals is to show that 
\emph{the presence of some nonuniform hyperbolicity implies 
the existence of a global dominated splitting.} 
It has been previously understood in \cite{BV}
that the presence of non-zero Lyapunov exponents implies the existence of ``local'' dominated splittings. 
On the other hand, global dominated splittings not only provide considerable restrictions on the dynamics (for instance, the topology of the ambient space is constrained),
but it is a basic starting point towards proving \emph{ergodicity}.

All known arguments ensuring frequent ergodicity
require at least a dominated splitting: see for example \cite{PS}, \cite{Tah}, \cite{ABW}, \cite{RRTU}.
In fact, stably ergodic diffeomorphisms necessarily have a global dominated splitting \cite{AM}.

In the $C^1$-generic situation, despite being unable to obtain full ergodicity,
we show that the nonuniformly hyperbolic part of the space forms a ergodic component.


\medskip

The result of \cite{BV} is based on ideas of Ma\~n\'e \cite{Mane_ICM}, 
who suggested that for generic diffeomorphisms the measurable and asymptotic information provided by the Oseledets theorem could be improved to continuous and uniform.
In a similar spirit, 
we study how regularly 
certain measurable objects (invariant sets, the ergodic decomposition, and Lyapunov exponents)
vary with respect to the dynamics,
obtaining improved properties in the generic case.
Later we combine this information
with an arsenal of $C^1$ tools and some Pesin theory
(especially the recent work \cite{RRTU})
to address the existence of global dominated splittings
and ergodicity of the nonuniformly hyperbolic set.
Since (most of) Pesin theory requires more than $C^1$ differentiability,
our arguments use the smoothing result of \cite{Avila}.

We proceed now to a formal statement of our main results.

\subsection{A Generic Dichotomy}

Let $M$ be a smooth compact connected manifold of dimension at least $2$, 
and let $m$ be a smooth volume measure, that we also call Lebesgue.
Let $\Diff_m^r(M)$ be the set of $m$-preserving $C^r$-diffeomorphisms endowed with the
$C^r$ topology.


Let  $f \in \Diff_m^1(M)$.
By Oseledets theorem,
for $m$-almost every point $x \in M$
there is a splitting
$T_x M = E^1(x) \oplus \cdots \oplus E^{\ell(x)}(x)$,
and there are numbers $\hat\lambda_1(f,x) > \cdots > \hat\lambda_{\ell(x)}(f,x)$,
called the Lyapunov exponents,
such that
$$
\lim_{n \to \pm \infty} \frac{1}{n} \log \big\| Df^n(x) \cdot v \big\| = \hat \lambda_i(f,x)
\quad \text{for every $v \in E_i(x) \setminus \{0\}$.}
$$
Repeating each Lyapunov exponent $\hat\lambda_i(f,x)$ according to its multiplicity $\dim E^i(x)$,
we obtain a list $\lambda_1(f,x) \ge \lambda_2(f,x) \ge \cdots \ge \lambda_d(f,x)$.
Since volume is preserved, $\sum_{j=1}^d \lambda_j(f,x) = 0$.

A point $x$ (or its orbit) is called \emph{nonuniformly hyperbolic} if
all its Lyapunov exponents are non-zero.
The set of those points is indicated by  $\NUH(f)$.

\begin{thm} \label{t.main}
There is a residual set $\cR \subset \Diff^1_m(M)$ such that for every  $f \in \cR$, 
either $m(\NUH(f)) = 0$ or
the restriction $f|\NUH(f)$ is ergodic and the orbit of almost every
point in $\NUH(f)$ is dense in the manifold.
\end{thm}

Let us now explain how Theorem~\ref{t.main} can be used to
construct \emph{global dominated splittings}.

It was shown by Bochi and Viana \cite{BV} that
\emph{for a generic $f$ in $\Diff^1_m(M)$, the Oseledets splitting along $m$-almost every orbit is either trivial or dominated.}
This means that for almost every $x \in M$, 
\begin{enumerate}
\item \label{i.trivial}
either $\ell(x)=1$, that is, all Lyapunov exponents are zero;
\item \label{i.dominated} 
or $\ell(x) > 1$ and there exists $n \geq 1$ such that
$$
\frac{\|Df^m( f^k x) \cdot v_i \|}{\|Df^m(f^k x) \cdot v_j \|} > 2 
$$
for every $k \in \Z$, unit vectors $v_i \in E^i(f^k x)$, $v_j \in E^j(f^k x)$ with
$i<j$, and $m \geq n$.
\end{enumerate}
In particular, the manifold $M$ equals $Z \sqcup \Lambda$ mod~$0$,
where $Z$ is the set where all Lyapunov exponents are zero,
and $\Lambda$ is an increasing union of Borel
sets $\Lambda_n$ where the Oseledets splitting is nontrivial and
dominated with uniform~$n$.
Since dominated splittings are always uniformly continuous
(see e.g.~\cite{BDV_book}),
there is a (uniform, nontrivial) dominated splitting over the
closure of each $\Lambda_n$, 
though not necessarily over the closure of $\Lambda$.  

Thus, as a direct consequence of \cite{BV} and Theorem~\ref{t.main}, we
get:

\begin{corol} 
There is a residual set $\cR \subset \Diff^1_m(M)$ such that for every $f \in \cR$, 
either $m(\NUH(f))=0$
or there is a global dominated splitting.
\end{corol}

Sometimes there are topological obstructions to the existence of global
dominated splittings.  For example, since the tangent bundle of even
dimensional spheres admits no non-trivial
invariant subbundle\footnote{Suppose the sphere $S^{2k}$
has a non-trivial field $E$ of $k$-planes, with $0<n<2k$.
Using that $S^{2k}$ is simply connected, we can continuously orient the planes.
Thus the Euler class $e(E)$ is well-defined in $H^n(S^{2k};\Z) = \{0\}$.
Let $F$ be the field of $(2k-n)$-planes orthogonal to $E$,
oriented so that $TS^{2k} = E \oplus F$.
Then $2 = e(TS^{2k}) = e(E) \smile e(F) = 0 \smile 0 = 0$,
contradiction.
We thank Daniel Ruberman for explaining this to us.},
the corollary implies that for the
generic $f \in \Diff^1_m(S^{2k})$,
there is at least one zero Lyapunov exponent at almost every point.
(For $k=1$ this follows from the Ma\~n\'e--Bochi Theorem \cite{B_ETDS}.)

Let us remark that in the \emph{symplectic} case
a stronger statement holds:
$C^1$-generic symplectomorphisms 
are either ergodic and Anosov or have at least two zero Lyapunov exponents at almost every point;
see \cite{B_JIMJ}.

\subsection{More New Generic Properties}

As mentioned before, the proof of Theorem \ref {t.main} depends on
some new results about the regularity of the
dependence of certain measurable objects with respect to the dynamics.

The most basic and abstract of such results (Theorem~\ref{t.CED})
shows that \emph{a generic $f$ in $\Diff^r_m(M)$ is a continuity point of the
ergodic decomposition of Lebesgue measure.} 
More precisely, if $f$ is generic then for every $C^r$-nearby map $g$,
the ergodic decompositions of $m$ with respect to $f$ and $g$
are close.

This result will be used to show, for $C^r$-generic $f$:
\begin{itemize}
\item the existence of appropriate continuations of invariant sets
(Theorem~\ref{t.PIS}),
\item the $L^1$-continuity of certain dynamically defined
functions, in particular
the Lyapunov exponents (Theorem~\ref{t.CLS}).
\end{itemize}
These theorems work for any $r \ge 1$ (and, in a certain sense, also for $r=0$),
and even for other measures (see Remark~\ref{r.other measures}),
and we believe they have independent interest.

\subsection{Main Ideas of the Proof of Theorem~\ref{t.main}}\label{ss.outline}

Let us explain the proof of the main result in a brief and simplified way.

Let $f$ be a $C^1$-generic volume-preserving diffeomorphism.
Assume that $\NUH_i(f) = \{\lambda_i > 0 > \lambda_{i+1}\}$ has positive measure for some $i$.
By \cite{BV}, we can take a Borel subset $\Lambda \subset \NUH_i(f)$ 
with $m \left( \NUH_i(f) \setminus \Lambda\right) \ll 1$
where the splitting 
that separates positive from negative Lyapunov exponents is (uniformly) 
dominated.

Despite $f$ being only $C^1$, domination allows us to find Pesin manifolds
for the points on $\Lambda$.
More precisely, there are certain non-invariant sets $\block(f,\ell)$, called \emph{Pesin blocks},
such that if $x\in \block(f,\ell)$ then the Pesin manifolds $W^s(x)$ and $W^u(x)$ have ``size'' at least $r(\ell)$;
moreover $m \left( \Lambda \setminus \block(f,\ell) \right) \to 0$ and $r(\ell)\to 0$ as $\ell\to \infty$. 
The Pesin blocks are explicitly defined in terms of certain Birkhoff sums,
so it will be possible later to control how they vary with the diffeomorphism.

We fix $\ell$ large and $0 < r \ll r(\ell)$.
We then find 
a hyperbolic periodic point $p$ such that 
the ball $B(p,r)$ has a positive measure intersection with the Pesin block $\block(f,\ell)$,
and $p$ itself is also in $\block(f,\ell)$.
In order to find such $p$ we use 
an improved version of the Ergodic Closing Lemma due to \cite{ABC}.\footnote{Although
the situation is similar to Katok's Closing Lemma, we don't use it.}

We consider the \emph{Pesin heteroclinic class of $p$},
a concept introduced in the paper \cite{RRTU}.
It is the set $\class(p, f)$ of the points $x\in M$ whose Pesin manifolds $W^u(x)$ and $W^s(x)$
intersect respectively $W^s(\cO(p))$ and $W^u(\cO(p))$ in a transverse way.
In our situation, 
the class $\class(p, f)$ has positive measure,
because Pesin manifolds are much longer than $r$ for points in the block.

Using the new generic properties (Theorems~\ref{t.PIS} and \ref{t.CLS})
it is possible to show that the situation is \emph{robust}:
For any $g$ sufficiently close to $f$, the new Pesin block $\block(g,\ell)$ is close
to the old one,  
the continuation $p^g$ of the periodic point $p$ belongs to $\block(g,\ell)$,
and the ball $B(p^g,r)$ has a positive measure intersection with $\block(g,\ell)$.
In particular, the new Pesin heteroclinic class $\class(p^g, g)$ has positive measure.

Using \cite{Avila}, we choose a $C^2$ volume-preserving diffeomorphism $g$ close to $f$.
This permits us to apply the ergodicity criterion from \cite{RRTU}
and conclude that $g$ restricted to $\class(p^g, g)$ is ergodic.

We get an ergodic component for the the original map $f$ 
using that its ergodic decomposition varies continuously (Theorem~\ref{t.CED}).
We are able to show that this component is in fact $\NUH_i(f)$.
To show that this set has a positive measure intersection with
every nonempty open set in the manifold,
we use the $C^1$-generic property that stable manifolds of periodic points are dense.
Another $C^1$-generic property says that every pair of periodic points is homoclinically connected,
and using this we can show the index $i$ is unique.

\subsection{Questions}

We still don't understand well ergodic properties of $C^1$-generic volume-preserving diffeomorphisms.
Even in dimension $2$ the picture is incomplete:
By \cite{B_ETDS}, the generic diffeomorphism is either Anosov or has zero 
Lyapunov exponents almost everywhere;
but we don't know much about the dynamics in the second alternative -- are those maps\footnote{By
\cite{BC}, there are points with dense orbits, but we don't know if they form a positive measure set.} 
ergodic, for example?

Perhaps we may separate the more familiar nonuniformly hyperbolic world from the unexplored world of all zero exponents. 
Optimistically, we conjecture that generically zero exponents cannot appear 
along with nonzero exponents in a positive measure set.
In view of our results, this question can be posed as follows:

\begin{conj}
For generic $f\in \Diff_m^1(M)$, 
either $f$ has all exponents zero at Lebesgue almost every point,
or $f$ is ergodic and \emph{nonuniformly Anosov}\footnote{This term was coined by Martin Andersson.},
that is, nonuniformly hyperbolic with a global
dominated splitting separating the positive exponents
from the negative ones.
\end{conj}

M.~A.~Rodriguez-Hertz has announced a proof of this conjecture in dimension
$3$ which uses the results of this paper.

Notice that the conjecture is false in the symplectic case:
there are nonempty open sets $\cU$ of partially hyperbolic symplectomorphisms that are not Anosov,
and it is shown in \cite{B_JIMJ} that for generic maps in $\cU$ 
the Lyapunov exponents along the center direction vanish.
Even so, it is possible to show that generic partially hyperbolic diffeomorphisms are ergodic:
see \cite{ABW}.

\subsection{Organization of the Paper}

The remaining of this paper is organized as follows: 
In Section~\ref{s.prelim} we collect a few measure-theoretic facts to be used throughout the paper. 
In Section~\ref{s.new}, we
state precisely and prove Theorems \ref{t.CED}, \ref{t.PIS}, \ref{t.CLS}.
Section~\ref{s.ingredients} contains more preliminaries:
\begin{itemize}
\item In \S\ref{ss.c1 old stuff}, we recall several results of the ``$C^1$-generic theory'' of conservative diffeomorphisms,
especially some from \cite{BC} and \cite{ABC}.
\item In \S\ref{ss.dom pesin}, we explain the ``$C^1$-dominated Pesin theory'', and give an useful
technical tool (Lemma~\ref{l.block}) to estimate the size of Pesin blocks.
This part does not use preservation of volume.
\item In \S\ref{ss.RRTU}, we recall the ergodicity criterion from \cite{RRTU}. 
\end{itemize}
Then in Section~\ref{s.proof} we give the proof of Theorem \ref{t.main}.  
As explained in \S\ref{ss.outline},
the regularity results of Section~\ref{s.new} are used repeatedly, basically
to allow us to tie the $C^1$ and $C^2$ worlds through continuity.

\begin{ack}
We benefited from talks with F.~Abdenur, F.\ and M.~A.\ Rodriguez-Hertz,
and A.~Wilkinson.
We thank a referee for various corrections and suggestions.
\end{ack}

\section{Measure-Theoretic Preliminaries} \label{s.prelim}

\subsection{The Space of Probability Measures}

If $X$ is any compact Hausdorff space, we let $\cM(X)$ be the set of Borel probability measures on $X$,
endowed with the usual weak-star topology.
This is a Hausdorff compact space itself.
In  particular we may consider the space $\cM(\cM(X))$,
whose elements will be indicated by bold greek letters.

A fact that we will use several times is that if a sequence $\mu_n \to \mu$ in $\cM(X)$ 
then
$\mu(Y) \le \liminf \mu_n(Y)$ if $Y$ is open;
$\mu(Y) \ge \limsup \mu_n(Y)$ if $Y$ is closed;
$\mu(Y) =   \lim    \mu_n(Y)$ if $Y$ is a Borel set with $\mu(\partial Y) = 0$.

\subsection{Measure-Valued integration} 

\begin{propdef}
Let $(Y, \cY, \lambda)$ be a probability space
and $(Z, \cZ)$ be a measurable space.
Let $\mu_y$ be probability measures on $(Z,\cZ)$,
defined for $\lambda$-almost every $y \in Y$.
Suppose that 
\begin{equation}\label{e.measurability}
	\text{for each $B \in \cZ$, the function $y \in Y \mapsto \mu_y(B) \in \R$ is $\cY$-measurable.}
\end{equation}
Then there is a unique probability measure $\bar \mu$ on $(Z,\cZ)$ such that 
for any bounded $\cZ$-measurable function $\phi: Z \to \R$, we have\footnote{It 
is part of the statement that the integrand $y \mapsto \int \phi \,  d\mu_y$ is measurable.}
\begin{equation}\label{e.integration}
\int_Y \int_Z \phi(z) \, d\mu_y(z) \, d\lambda(y) = \int_Z  \phi(z) \, d\bar\mu(z) .
\end{equation}
We call $\bar \mu$ the \emph{integral}
of the function $y \to \mu_y$,
and we indicate
$$
\bar \mu = \int \mu_y \, d\lambda(y) \, .
$$
\end{propdef}

\begin{proof}
By the ``skew'' Fubini theorem from~\cite{Johnson},
there is a measure $\rho$ on $(Y \times Z, \cY \times \cZ)$ such that
$$
\int_Y \int_Z \psi(y,z) \, d\mu_y(z) \, d\lambda(y) = \int_{Y \times Z} \psi(y,z) \, d\rho(y,z) 
$$
for any bounded $\cY \times \cZ$-measurable function $\psi$.\footnote{Observe that
a converse to this result is related to Rokhlin Desintegration Theorem
\cite[\S C.6]{BDV_book}.}
We define $\bar \mu$ as the push-forward of $\rho$ by the projection $Y\times Z \to Z$.
\end{proof}

Let us observe a few properties of the integral for later use:
\begin{itemize}
\item The integral behaves well under push-forwards.
More precisely, if $W$ is another measurable space and $F: Z \to W$ is a measurable map
then $F_* \bar\mu = \int F_* \mu_y \, d\mu(y)$.

\item
Formula \eqref{e.integration}
also holds for $\bar \mu$-integrable functions $\phi$.
More precisely, if $\phi \in L^1(\bar \mu)$ (so $\phi$ is an equivalence class of functions)
then $y \mapsto \int \phi \, d\mu_y$ is a well-defined element of $L^1(\lambda)$
whose integral is given by \eqref{e.integration}.
This follows easily from the Monotone Convergence Theorem.
\end{itemize}

Another observation is that an integral can be approximated by
finite convex combinations:

\begin{lemma}\label{l.combination}
In the situation above, 
assume in addition that $Z$ is a compact Hausdorff space and $\cZ$ is the Borel $\sigma$-algebra.
Then for any neighborhood $\cN$ of $\bar \mu = \int \mu_y \, d\lambda(y)$ in $\cM(Z)$,
there exist $y_1$, \ldots, $y_k \in Y$ and positive numbers $c_1$, \ldots, $c_k$ with $\sum c_i = 1$
such that the measure $\sum c_i \mu_{y_i}$ belongs to $\cN$.
\end{lemma}

\begin{proof}
We can assume that the neighborhood of $\bar \mu$ is of the form
$$
\cN = \left\{\nu \in \cM(Z) ; \; \left|\int \phi_j \, d\nu - \int \phi_j \, d\bar\mu  \right|< \eps 
\  \forall j=1,\ldots,n \right\} ,
$$
for some $\eps>0$ and continuous functions $\phi_1$, \ldots, $\phi_n$.
Define $\Phi_j : Y \to \R$ by $\Phi_j(y) = \int \phi_j \, d\mu_y$.
Since those functions are bounded and measurable, we can approximate them by simple functions.
Take a measurable partition $Y=E_1 \sqcup \cdots \sqcup E_k$
and numbers $a_{ij} \le b_{ij}$ (where $1\le i \le k$, $1 \le j \le n$)
such that for each $j$,
$$
\sum_i a_{ij} \charac_{E_i} \le \Phi_j \le \sum_i b_{ij} \charac_{E_i} 
\quad \text{and} \quad
\sum_i (b_{ij} - a_{ij}) \lambda(E_i) < \eps. 
$$
Define $c_i = \lambda(E_i)$ and choose points $y_i \in E_i$.
Since $\int \Phi_j \, d\lambda = \int \phi_j \, d\bar\mu$,
it follows that the measure $\sum c_i \mu_{y_i}$ belongs to $\cN$.
\end{proof}

Let us check the measurability condition~\eqref{e.measurability}
in the case that the function $y \to \mu_y$ is the identity:

\begin{lemma}\label{l.measurable id}
Let $Z$ be a compact Hausdorff space. 
Then for every Borel set $B \subset Z$, 
the function $\mu \in \cM(Z) \mapsto \mu(B) \in \R$ is Borel-measurable.
\end{lemma}

\begin{proof}
If $B$ is an open set, 
there exists a sequence of continuous functions $\phi_n$ 
(namely, a sequence that converges pointwise to $\charac_{B}$)
such that the 
function $\mu \mapsto \mu(B)$ is the pointwise limit of the sequence of continuous functions 
$\mu \mapsto \int \phi_n \, d\mu$, and so it is measurable.
That is, the class of Borel sets $B \subset Z$ such that $\mu  \mapsto \mu(B)$ is measurable
includes all open sets.
This class is evidently closed under nested intersection,
and thus it contains all $G_\delta$ sets.
Now, if $B \subset Z$ is any Borel set then there exists a $G_\delta$ set $\tilde B \supset B$
such that $\mu(\tilde B) = \mu(B)$ for every Borel measure $\mu$,
and thus we are done.
\end{proof}

Thus if $Z$ is a compact Hausdorff space
and $\blambda \in \cM(\cM(Z))$ then 
$\mu = \int \nu \, d\blambda(\nu)$ is a well-defined element of $\cM(Z)$.
We say that \emph{$\blambda$ is a decomposition of $\mu$.}

\subsection{Ergodic Decomposition}

Let $f: X \to X$ is a continuous map on a compact metric space $X$.
We let $\cM(f) \subset \cM(X)$ denote the set of $f$-invariant probabilities;
and let $\cM_\mathrm{erg}(f) \subset \cM(f)$ denote the set of $f$-ergodic probabilities.
Both $\cM(f)$ and $\cM_\mathrm{erg}(f)$ are Borel subsets: the former is closed 
and the latter is a $G_\delta$.\footnote{See \cite[Prop.~1.3]{Phelps}.}

Given $\mu \in \cM(f)$,
we let $\bkappa_{f,\mu} \in \cM(\cM(X))$ 
be the \emph{ergodic decomposition}
of the measure $\mu$;
that is, the unique decomposition of $\mu$
such that $\bkappa_{f,\mu} (\cM_\mathrm{erg}(f)) = 1$.

According to \cite{Mane_book},
ergodic decompositions can be obtained as follows.\footnote{Most other proofs 
of the existence of the ergodic decomposition
are more abstract and rely on Choquet's theorem; see~\cite{Phelps}.} 
There exists a Borel subset $R_f \subset X$ that has full $\mu$ measure with respect to any $\mu \in \cM(f)$
such that for any $x \in R_f$, the measure
\begin{equation}\label{e.statistics}
\beta_x = \beta_{f,x} = \lim_{n\to \infty} \frac{1}{n} \sum_{j=0}^{n-1} \delta_{f^j x}
\quad \text{exists and is $f$-ergodic.}
\end{equation}
Then for any $\mu\in \cM(f)$, its push-forward by the
(evidently measurable) map $x\mapsto \beta_x$ is the ergodic decomposition $\bkappa_{f,\mu}$;
that is, $\bkappa_{f,\mu}(\cU) = \mu(\beta^{-1}(\cU))$
for any Borel set $\cU \subset \cM(X)$.
As an immediate consequence of Lemma~\ref{l.measurable id},
we obtain that the function $\mu \mapsto \bkappa_{f,\mu}$
satisfies the the measurability condition~\eqref{e.measurability}:

\begin{lemma} \label{l.measurable decomposition}
For any Borel set $\cU \subset \cM(X)$, the function
$\mu \in \cM(f) \mapsto \bkappa_{f,\mu}(\cU) \in \R$ is measurable.\footnote{However, the function
is not necessarily continuous: in general $\cM_\mathrm{erg}(f)$ is not closed,
and if $\mu_0$ is a non-ergodic accumulation point of ergodic measures 
then for some choice of $\cU$ the function $\mu \mapsto \bkappa_{f,\mu}(\cU)$ is not continuous at $\mu_0$.}
\end{lemma}

\section{Some New Generic Properties} \label{s.new}

\subsection{Generic Continuity of the Ergodic Decomposition}\label{ss.CED}

For an integer $r\ge 1$, let $\Diff^r(M)$ be the set of  $C^r$ diffeomorphisms with the $C^r$ topology.
Let also $\Diff^0(M)$ be the set of homeomorphisms, with the topology under which $f_n \to f$
if $f_n^{\pm 1} \to f^{\pm 1}$ uniformly.
Let $\Diff_m^r(M)$ be the set of elements of $\Diff^r(M)$ that leave invariant the measure $m$.
Then $\Diff^r(M)$ and $\Diff^r_m(M)$ are Baire spaces for any integer $r \ge 0$.

Recall that if $f \in \Diff^r(M)$ and $\mu \in \cM(f)$,
then $\bkappa_{f,\mu} \in \cM(\cM(M))$ 
indicates the ergodic decomposition of the measure $\mu$.
Most of the time we will work with diffeomorphisms $f$ that preserve Lebesgue measure $m$,
and we abbreviate $\bkappa_f = \bkappa_{f,m}$.

\begin{thm}
\label{t.CED}
Fix an integer $r\ge 0$.
The points of continuity of the map
$$f \in \Diff_m^r(M) \mapsto \bkappa_{f} \in \cM(\cM(M))$$
form a residual subset.
\end{thm}

To get a taste for this result, consider the circle case.
Then the points of continuity of the ergodic decomposition are precisely
the irrational rotations and the orientation-reversing involutions.


\begin{proof}[Proof of Theorem~\ref{t.CED}]
Let $\phi:M \to \R$ is a continuous map with $\int \phi \, dm = 0$.
For each $f \in \Diff_m^r(M)$, let $\phi_{f,n} = \phi + \cdots + \phi \circ f^{n-1}$ and
$\hat\phi_{f} = \lim_{n \to \infty} \phi_{f,n} / n$.
The sequence $\|\phi_{f,n}\|_{L^2(m)}$ is subadditive,
and therefore 
$$
\| \hat\phi_{f} \|_{L^2(m)} = \inf_n \frac{\|\phi_{f,n}\|_{L^2(m)}}{n}   \, .
$$
On the other hand, $f \mapsto \phi_{f,n}$ is 
a continuous map from $\Diff_m^r(M)$ to $L^2(m)$.
In particular, we conclude that
\emph{the function $f \in \Diff_m^r(M) \mapsto \| \hat\phi_{f} \|_{L^2(m)}$
is upper-semicontinuous.}
In particular its points of continuity form a residual subset $\cR_\phi$ of 
$\Diff_m^r(M)$.

Let $L^2_0(m) = \big\{ \phi \in L^2(m) ; \; \int \phi \, dm = 0 \big\}$.	
Take a countable dense subset $\{\phi_j\}$ of $L^2_0(m)$ formed by continuous functions.
Define a residual subset $\cR = \bigcap_j \cR_{\phi_j}$.
To prove the theorem, we will show that each $f$ in $\cR$ is a point of continuity
of the ergodic decomposition.

To begin, notice that
$$
\| \hat\phi_{f} \|_{L^2(m)}^2 = \int_{\cM(M)} \left( \int \phi \, d\mu \right)^2 \, d\bkappa_{f}(\mu) \, .
$$
Let $\cD(m) \subset \cM(\cM(M))$ be the set of decompositions of Lebesgue measure~$m$.
We define the \emph{variance} of a function $\phi \in L^2_0(m)$ with respect to 
a decomposition $\blambda \in \cD(m)$ as 
$$
\Var(\phi,\blambda) = \int_{\cM(M)} \left( \int \phi \, d\mu \right)^2 \, d\blambda(\mu) \, .
$$
Thus $\Var(\phi,\bkappa_f) = \| \hat\phi_{f} \|_{L^2(m)}^2$.

\begin{lemma}\label{l.var cont}
$\Var(\phi,\blambda)$ is finite and depends continuously on $\phi$ and $\blambda$.
\end{lemma}

\begin{proof}
Given $\phi \in L^2_0(m)$, let $\Phi(\mu) = \int \phi \, d\mu$.
By convexity, $\Phi(\mu)^2 \le \int \phi^2 \, d\mu$. 
Integrating with respect to $\blambda$, we get
$$
\Var(\phi, \blambda) = \|\Phi\|_{L^2(\blambda)}^2 \le \int \phi^2 \, dm = \|\phi\|_{L^2(m)}^2 < \infty.
$$
Hence the triangle inequality in $L^2(\blambda)$ gives
$$
\Var^{1/2}(\phi + \psi, \blambda) \le \Var^{1/2}(\phi,\blambda) +\Var^{1/2}(\psi,\blambda) \quad
\text{for all $\phi$, $\psi \in L^2_0(m)$.}
$$
Therefore
$$
\big| \Var^{1/2}(\phi,\blambda) - \Var^{1/2}(\psi,\blambda) \big| \le \| \phi - \psi \|_{L^2(m)} \quad
\text{for all $\phi$, $\psi \in L^2_0(m)$.}
$$
and so the functions $\phi \in L^2(m) \mapsto \Var^{1/2}(\phi,\blambda)$ with $\blambda \in \cD(m)$
form a uniformly equicontinuous family.
Finally, the functions $\blambda \in \cD(m) \mapsto \Var(\phi,\blambda)$ are continuous:
this is obvious if $\phi$ is continuous and the general case follows by equicontinuity.
\end{proof}

Given $f \in \Diff_m^r(M)$, let $\cD(m,f)$ be the set of
decompositions of $m$ into (non-necessarily ergodic) $f$-invariant measures,
that is, the set of $\blambda \in \cD(m)$
such that $\blambda (\cM(f)) = 1$.

\begin{lemma}\label{l.var max}
Let $f\in \Diff_m^r(M)$.
Then for any $\phi \in L^2_0(m)$ and $\blambda \in \cD(m,f)$,
$$
\Var(\phi,\blambda) \le \Var(\phi,\bkappa_{f}) \, .
$$
Moreover, if $\blambda \in \cD(m,f)$ is such that the equality holds for every $\phi\in L^2_0(m)$
then  $\blambda = \bkappa_{f}$.
\end{lemma}

\begin{proof}
Fix $f$ and $\blambda \in \cD(m,f)$.
The measure $\int \bkappa_{f,\mu} \, d\blambda (\mu)$
(which makes sense by Lemma~\ref{l.measurable decomposition})
is a decomposition of $m$
and gives full weight to $\cM_\mathrm{erg}(f)$;
therefore it equals~$\bkappa_{f}$.
So for any $\phi\in L^2_0(m)$ we have
\begin{alignat*}{2}
\Var(\phi,\blambda)
&=  \int \left( \int \phi \, d\mu \right)^2 \, d\blambda(\mu)                           \\
&=  \int \left( \iint \phi \, d\nu  \, d\bkappa_{f,\mu}(\nu) \right)^2 \, d\blambda(\mu)  \\
&\le\iint \left( \int \phi \, d\nu \right)^2  d\bkappa_{f,\mu}(\nu) \, d\blambda(\mu)    &\quad&\text{(by convexity)}\\
&=  \int \left( \int \phi \, d\nu \right)^2  d\bkappa_{f}(\nu) &\quad&\text{(since $\bkappa_{f} = \int \bkappa_{f,\mu} \, d\blambda (\mu)$\,)}\\
&= \Var(\phi,\bkappa_{f}) \, .
\end{alignat*}
This proves the first part of the lemma.
If equality holds above then
for $\blambda$-almost every $\mu$,
the function $\nu \mapsto \int \phi \, d\nu$
is constant $\kappa_{f,\mu}$-almost everywhere.
By averaging, that constant must be $\int \phi \, d\mu$.
Now assume that this happens say for every $\phi$ in a countable dense subset $D$ of $L_0^2(m)$.
Then for $\blambda$-almost every $\mu$ and $\kappa_{f,\mu}$-almost every $\nu$,
we have
$\int \phi \, d\nu = \int \phi \, d\mu$ for all $\phi$ in $D$.
Hence for $\blambda$-almost every $\mu$, the measure $\bkappa_{f,\mu}$ is the Dirac
mass concentrated on $\mu$,
and in particular $\mu$ is ergodic.
Since $\bkappa_{f}$ is the only decomposition of $m$ giving full weight to $\cM_\mathrm{erg}(f)$,
we have $\blambda=\kappa_f$.
\end{proof}

We now complete the proof of Theorem~\ref{t.CED}.
Fix $f \in \cR$.
Take any sequence $f_n \to f$.
such that $\bkappa_{f_n}$ has a limit $\blambda$.
Recall that $\{\phi_j\}$ is a dense subset of $L_0^2(m)$.
For each $j$ we have 
\begin{alignat*}{2}
\Var(\phi_j, \blambda) 
&= \lim_{n \to \infty} \Var(\phi_j , \bkappa_{f_n})  
&\quad&\text{(by Lemma~\ref{l.var cont})}\\
&= \Var(\phi_j, \bkappa_{f}) 
&\quad&\text{(since $f$ is a point of continuity of $g \mapsto \Var(\phi_j , \bkappa_{g})$).}
\end{alignat*}
By Lemma~\ref{l.var cont} again it follows that $\Var(\phi, \blambda) = \Var(\phi, \bkappa_{f})$
for every $\phi \in L^2_0(m)$.
By Lemma~\ref{l.var max}, we have $\blambda = \bkappa_{f}$.
\end{proof}

\subsection{More on the Ergodic Decomposition} \label{ss.more}

We may informally interpret Theorem~\ref{t.CED} as follows:
Given a $m$-preserving diffeomorphism $f$, consider 
the proportion (with respect to $m$) of points in the manifold 
whose $f$-orbits have approximately a certain prescribed statistics;
then for generic $f$ this proportion does not change much if $f$ is perturbed.
Let us improve this a little and show that if $f$ is perturbed
then the \emph{set} of points whose orbits have approximately a certain prescribed statistics
does not change much.

If $\cU \subset \cM(M)$ is a Borel set,
let $X_{\cU,f}$ be
the set of all $x \in M$ such that 
$\frac {1} {n} \sum_{j=0}^{n-1} \delta_{f^j(x)}$ converges and belongs to $\cU$.
That is, $X_{\cU,f} = \beta_f^{-1}(\cU)$, where $\beta_f$ is given by \eqref{e.statistics}.
Then $X_{\cU,f}$ is an $f$-invariant Borel set, and
$\mu(X_{\cU,f}) = \bkappa_{f, \mu}(\cU)$ for any $\mu \in \cM(f)$.

\begin{lemma}\label{l.sunday}
If $f_k \to f$ and $\bkappa_{f_k} \to \bkappa_{f}$ then
for any open set $\cU \subset \cM(M)$ with $\bkappa_{f}(\partial \cU) = 0$ we have
$$
\lim_{k \to \infty}
m \big( X_{\cU,f_k} \vartriangle X_{\cU,f} \big) = 0 \, .
$$
\end{lemma}

This lemma gives a precise meaning to the informal discussion above.

\begin{proof}
The hypotheses imply that $m(X_{\cU,f_k}) \to m(X_{\cU,f})$. 
We may thus assume that there exists $c>0$ such that
$Y_k=X_{\cU,f} \setminus X_{\cU,f_k}$ satisfies $m(Y_k)>c$ for all $k$.
Let 
$$
\mu_k = \frac{1}{m(Y_k)} \lim_{n\to \infty} \frac{1}{n} \sum_{j=0}^{n-1} (f_k)^j_*(m|Y_k),
$$
that is, $\mu_k$ is the probability measure that is absolutely continuous with respect to $m$
and has density
$$
\frac{1}{m(Y_k)} \lim_{n\to \infty} \frac{1}{n} \sum_{j=0}^{n-1} \charac_{Y_k} \circ f_k^j \, .
$$
Then $\mu_k$ is $f_k$-invariant.
We claim 
that
$$
\mu_k(Y_k) \ge c \, .
$$
Indeed, let $P: L^2(m) \to L^2(m)$ be the orthogonal projection onto the space of $f_k$-invariant functions,
and $\phi = \charac_{Y_k}$;
then, by Von Neumann's Ergodic Theorem (and using that $1$ is in the image of $P$), we have
$$
m(Y_k) \mu_k(Y_k) = \langle P \phi, \phi \rangle =  \langle P \phi, P \phi \rangle \ge  \langle P \phi, 1 \rangle^2 
= \langle \phi, 1 \rangle^2 = m(Y_k)^2 \, ,
$$
so $\mu_k(Y_k) \ge m(Y_k) > c$.

By passing to a subsequence, we assume that $\mu_k$ has a limit $\mu$, which is 
evidently $f$-invariant.
Since each $\mu_k$ is absolutely continuous with density bounded by $c^{-1}$, 
the same is true for~$\mu$.
It also follows from the uniform bounds on densities that 
$\mu(X_{\cU,f}) = \lim_{k \to \infty} \mu_k(X_{\cU,f}) \ge c$.

The definition of $Y_k$ implies that 
$\bkappa_{f_k, \mu_k}$ gives no weight to $\cU$.  
We claim that 
$\bkappa_{f_k, \mu_k} \to \bkappa_{f,\mu}$.
Because $\cU$ is open, this implies
$$
\bkappa_{f,\mu} (\cU) \le \liminf \bkappa_{f_k,\mu_k}(\cU) = 0,
$$ 
which contradicts $\mu(X_{\cU,f}) \geq c$.

To see the claim, notice that $\bkappa_{f_k, \mu_k}$ is absolutely continuous with
respect to $\bkappa_{f_k}$ with density at most $c^{-1}$;
indeed, for any Borel set $\cB \subset \cM(\cM(M))$, we have
$$
\bkappa_{f_k, \mu_k} (\cB) = \mu_k(X_{\cB, f_k}) \le c^{-1} m(X_{\cB, f_k}) = c^{-1} \bkappa_{f_k} (\cB).
$$
We may assume that $\bkappa_{f_k, \mu_k}$ has a limit $\blambda$.
Then $\blambda$ is absolutely continuous with respect to $\bkappa_{f}$;
indeed, for for any continuous $\Phi \ge 0$, we have
$$
\int \Phi \, d\blambda = \lim \int \Phi \, d\bkappa_{f_k, \mu_k} \le c^{-1} \lim \int \Phi \, d\bkappa_{f_k} 
= c^{-1} \int \Phi \, d\bkappa_{f} \, .
$$
In particular, $\blambda (\cM(M) \setminus \cM_\mathrm{erg}(f)) = 0$.
Moreover,
$$
\int \nu \, d\blambda(\nu) = \lim \int \nu \, d\bkappa_{f_k, \mu_k}(\nu) = \lim \mu_k = \mu \, . 
$$
So $\blambda$ is the $f$-ergodic decomposition of $\mu$.
This proves the claim and hence the lemma.
\end{proof}

We will show an interesting consequence of Lemma~\ref{l.sunday} (and Theorem~\ref{t.CED}),
although we won't use it directly. 

For $f\in \Diff_m^r(M)$, we look again at the map 
$\beta_f : R_f \subset M \to \cM(M)$ defined by \eqref{e.statistics}.
Let $\bnu_f$ be the measure 
concentrated on the graph of $\beta_f$ that projects on $m$,
that is, the push-forward of $m$ by the map $(\id, \beta_f)$. 

\begin{corol}
The points of continuity of the map $f \in \Diff_m^r(M) \mapsto \bnu_f \in \cM(M \times \cM(M))$
are the same as for the map $f \mapsto \bkappa_f \in \cM(\cM(M))$,
and in particular form a residual set.
\end{corol}

We omit the proof.

\subsection{Generic Persistence of Invariant Sets}

The next result says that if $f$ is a generic volume-preserving diffeomorphism,
then its measurable invariant sets persist in a certain (measure-theoretic and topological) sense
under perturbations of $f$.

If $\eta>0$ and $\Lambda \subset M$ is any set, let
$B_\eta(\Lambda)$ denote the $\eta$-neighborhood of~$\Lambda$, that is,
the set of $y \in M$ such that $d(y,x) < \eta$ for some $x \in \Lambda$.

\begin{thm}\label{t.PIS} 
Fix an integer $r\ge 0$. 
There is a residual set 
$\cR \subset \Diff_m^r(M)$ such that
for every $f\in \cR$, 
every $f$-invariant Borel set $\Lambda \subset M$, 
and every $\eta>0$,
if $g \in \Diff_m^r(M)$ is sufficiently close to $f$
then there exists a $g$-invariant Borel set $\tilde \Lambda$ such that
$$
\tilde\Lambda \subset B_\eta(\Lambda) \quad \text{and} \quad
m(\tilde \Lambda \vartriangle \Lambda) < \eta .
$$
\end{thm}

The proof will use Theorem~\ref{t.CED}, Lemma~\ref{l.sunday},
and a few other lemmas.

\begin{lemma}\label{l.tokyo}
If $f \in \Diff_m^r(M)$ and $\Lambda\subset M$ is a $f$-invariant Borel set,
then for any $\eta>0$ there exists an open subset $\cU$ of $\cM(M)$ such that
$\bkappa_{f}(\partial \cU) = 0$ and $m(\Lambda \vartriangle X_{\cU, f}) < \eta$.
\end{lemma}

\begin{proof}
Take a compact set~$K$ and an open set~$U$
such that $K \subset \Lambda \subset U$ and $m(U \setminus K) < \eta/6$.
Choose a continuous function $\phi:M \to \R$ such that $\charac_K \le \phi \le \charac_U$.
Given $a \in \R$, 
let $\cU$ be the set of measures $\mu \in \cM(M)$ such that $\int \phi \, d\mu > a$.
A moment's thought shows that $\partial \cU = \big\{ \mu \in \cM(M); \; \int \phi \, d\mu = a\big\}$.
Hence we can choose $a$ with $1/3 < a < 2/3$ so that  $\bkappa_{f}(\partial \cU) = 0$.
Notice that $X_{\cU,f} = \{\hat{\phi} > a \}$ mod~$0$, 
where $\hat{\phi} = \lim \frac{1}{n}\sum_{j=0}^{n-1} \phi \circ f^j$.
Therefore
\begin{align*}
m(X_{\cU,f} \setminus \Lambda) &\le \int_{X_{\cU,f} \setminus \Lambda} 3\hat\phi
\le \int_{M\setminus \Lambda} 3\hat\phi= \int_{M\setminus \Lambda} 3\phi < \eta/2 \, ,
\\
m(\Lambda \setminus X_{\cU,f}) &\le \int_{\Lambda \setminus X_{\cU,f}} 3(1-\hat\phi)
\le \int_{\Lambda}  3(1-\hat\phi) = \int_{\Lambda} 3 (1-\phi) < \eta/2 \, ,
\end{align*}
so $m(\Lambda \vartriangle X_{\cU, f}) < \eta$, as desired.
\end{proof}

If $V$ is any set and $f \in \Diff_m^r(M)$, let us denote
$$
V_f = \bigcap_{n \in \Z} f^n(V) .
$$

\begin{lemma}\label{l.milano}
Fixed any open $V$ with $m(\partial V) = 0$,
the measure $m(V_f)$
varies upper semi-continuously with $f$.
\end{lemma}

\begin{proof}
The function being considered is the infimum of a sequence of continuous functions:
\[
m (V_f) = \inf_{k >0} m \left( \bigcap_{|n|<k} f^n(V) \right) \, . \qedhere
\]
\end{proof}

\begin{lemma} \label{l.london}
For any $f \in \Diff_m^r(M)$, any open set $V$ with $m(\partial V) = 0$, and any $\eta>0$,
if $g\in \Diff_m^r(M)$ is sufficiently close to $f$ then $m(V_g \setminus V_f) < \eta$.
\end{lemma}

\begin{proof}
Indeed, for any sequence $g_k \to f$ we have $\limsup V_{g_k} \subset (\cl V)_f$,
and thus $\limsup m(V_{g_k} \setminus V_f) \le m( \limsup (V_{g_k} \setminus V_f))  \le m(\partial V) = 0$.
\end{proof}

\begin{proof}[Proof of Theorem~\ref{t.PIS}]
Fix a countable family $\cC$ of open subsets of $M$,
each with a boundary of zero measure,
such that for any compact
set $K$
and any $\eta>0$, there exists $V \in \cC$ such that 
$K \subset V \subset B_\eta(K)$.
Let $\cR\subset\Diff_m(M)$ be the intersection of the residual set of Theorem~\ref{t.CED}
with the set of points of continuity of $f \mapsto m(V_f)$ over $V \in \cC$.
Due to Lemma~\ref{l.milano}, $\cR$ is a residual set.

Fix $f\in \cR$, a Borel $f$-invariant set $\Lambda$, and $\eta>0$.
Let $V \in \cC$ be such that $\cl \Lambda \subset V \subset B_{\eta/2}(\cl \Lambda)$.
Using Lemma~\ref{l.tokyo}, find an open set $\cU \subset \cM(M)$
such that $\bkappa_{f}(\partial \cU) = 0$ and $m(\Lambda \vartriangle X_{\cU, f}) < \eta/4$.
If $g$ is sufficiently close to $f$ then 
$m(X_{\cU,g} \vartriangle X_{\cU,f}) < \eta/4$, by Lemma~\ref{l.sunday},
$m(V_g \setminus V_f) < \eta/4$, by Lemma~\ref{l.london},
and $m(V_g) > m(V_f) - \eta/4$, by continuity.
Then
$$
m(V_f\setminus V_g) = m(V_f) + m(V_g\setminus V_f) - m(V_g) < \eta/2.
$$
Given $g$ as above, define $\tilde \Lambda = V_g \cap X_{\cU, g}$.
This is a $g$-invariant Borel set contained in $B_{\eta}(\Lambda)$.
Moreover,
$$
\tilde\Lambda \vartriangle \Lambda \subset 
(\Lambda \setminus V_g) \cup (\Lambda \vartriangle X_{\cU, g}) \subset
(V_f \setminus V_g) \cup (\Lambda \vartriangle X_{\cU, f})  \cup (X_{\cU, f} \vartriangle X_{\cU, g}) 
$$
and therefore $m(\tilde\Lambda \vartriangle \Lambda) < \eta$.
\end{proof}

\subsection{Generic Continuity of the Lyapunov Spectrum}\label{ss.CLS}

\begin{thm}
\label{t.CLS}
Fix an integer $r\ge 1$.
For each $i$, the points of continuity of the map
$$
\lambda_i: \Diff_m^r(M) \to L^1(m)
$$	
form  a residual subset. 
\end{thm}

The proof uses Theorem~\ref{t.CED} and Lemma~\ref{l.sunday}:

\begin{proof}
For any $\mu \in \cM(M)$ ($f$-invariant or not), we define 
$$
L_i(f,\mu) = \inf_n L_i^{(n)}(f,\mu) \quad \text{where} \quad
L_i^{(n)}(f,\mu) = \int \frac{1}{n} \log \| \wed^i Df^n \| \, d\mu \, .
$$
Then $L_i : \Diff_m^r(M) \times \cM(M) \to \R$ is an upper semi-continuous function.
Notice that if $\mu \in \cM(f)$ then
$$
L_i(f,\mu) = \int L_i(f,x) \, d\mu(x) \quad \text{where} \quad L_i(f,x) 
= \lambda_1(f,x) + \cdots + \lambda_i(f,x) \, .
$$

Let $\cR_i \subset \Diff_m^r(M)$ be the set of continuity points of the
map $L_i( \mathord{\cdot}, m)$.
Let $\cR$ be the residual set given by Theorem~\ref{t.CED}.
Fix any $f \in \cR_i \cap \cR$.
To prove the theorem, we will show that the map
$L_i: \Diff_m^r(M) \to L^1(m)$ is continuous on $f$.

Let $\eps>0$.
Choose $n$ such that $L_i^{(n)}(f,m) < \eps + L_i(f,m)$.
By continuity of $L_i^{(n)}( \mathord{\cdot}, \mathord{\cdot})$
and compactness of $\supp \bkappa_{f}$,
there are open sets $\cU \supset \supp \bkappa_{f}$ and $\cV \ni f$ such that
$$
g \in \cV, \ \mu \in \cU \ \Rightarrow \
|L_i^{(n)}(g,\mu) - L_i^{(n)}(f,\mu)| < \eps \, .
$$

Let $c_0 < \cdots < c_J$ be real numbers such that $c_j - c_{j-1} < \eps$ and
the set of $\mu \in \cU$ such that $L_i^{(n)}(f,\mu) \in (-\infty, c_0] \cup \{c_1, \ldots, c_{J-1}\} \cup [c_J, +\infty)$
has zero $\bkappa_{f}$-measure.
Define a $\bkappa_{f}$-mod~$0$ partition of $\cU$ in open sets
$\cU_j = \{ \mu \in \cU ; \; c_{j-1} < L_i^{(n)}(f,\mu) < c_{j} \}$, $j=1,\ldots, J$
(whose boundaries have zero $\bkappa_{f}$-measure).
Define a function $\Gamma:\cM(M) \to \R$ by $\Gamma = \sum_{j=1}^J c_j \mathbbm{1}_{\cU_j}$.

\begin{claim}
For every $g$ sufficiently close to $f$, we have
\begin{equation}\label{e.bla}
\int \big| L_i(g,\mu) - \Gamma(\mu) \big| \, d\bkappa_{g}(\mu) < O(\eps).
\end{equation}
\end{claim}
Indeed, since $f \in \cR$, if $g$ is close to $f$ then
the set $\cZ = \cM(M) \setminus \bigcup_{j=1}^J \cU_j$
has $\bkappa_{g}$-measure less than $\eps$.
On the other hand,
\begin{multline*}
\mu \in \cM(g) \cap \cU_j \ \Rightarrow \
L_i(g,\mu) \le L_i^{(n)}(g,\mu) < L_i^{(n)}(f,\mu) + \eps < c_j + \eps = \Gamma(\mu) + \eps \\
\Rightarrow \
\big| L_i(g,\mu) - \Gamma(\mu) \big| \le 2 \eps - L_i(g, \mu) + \Gamma(\mu)
< 3 \eps - L_i(g, \mu) + L^{(n)}_i(f, \mu)  \, .
\end{multline*}
Therefore, letting $C$ be an upper bound for $\log \|\wed^i Dg \|$ on a neighborhood of $f$,
and also for $|c_0|$, $|c_J|$,
we can write
\begin{align*}
\int \big| L_i(g,\mu) - \Gamma(\mu) \big| \, d\bkappa_{g}(\mu)
&\le 4C\bkappa_{g}(\cZ) + 3 \eps + \int \big[L^{(n)}_i(f, \mu) - L_i(g,\mu) \big] \, d\bkappa_{g}(\mu) \\
&\le (4C+3)\eps + L^{(n)}_i(f, m) - L_i(g, m) \\
&\le (4C+4)\eps + L_i(f, m)       - L_i(g, m)
\end{align*}
Since $f \in \cR_i$, $L_i(g, m)$ is close to $L_i(f, m)$
provided $g$ is close enough to $f$, thus completing the proof of \eqref{e.bla}.

Assume that $g$ is close enough to $f$ so that
$m(X_{\cU_j, g} \vartriangle X_{\cU_j, f}) < \eps$ for each~$j$
(see Lemma~\ref{l.sunday}).

We claim that the function $L_i(g)$ is close in $L^1(m)$ to $\sum_{j=1}^J c_j \mathbbm{1}_{X_{\cU_j,g}}$
and hence to  $\sum_{j=1}^J c_j \mathbbm{1}_{X_{\cU_j,f}}$ and hence to $L_i(f)$.
Indeed, the functions $L_i(g)$ and $\sum_{j=1}^J c_j \mathbbm{1}_{X_{\cU_j,g}}$
are both $g$-invariant, and it follows that their $L^1(m)$-distance is exactly the left hand side of \eqref{e.bla}.
\end{proof}

In fact, the proof above yields a more general result, which we now describe.
For each $f \in \Diff_m^r(M)$, let $\phi_{f,n}$, $n\in \Z_+$ be a 
sequence of continuous functions that is subadditive with respect
to $f$, that is,
$$
\phi_{f,k+n} \le \phi_{f,k} \circ f^n + \phi_{f,n} \, .
$$
Also assume that $f \mapsto \phi_{f,n} in C^0(M,R)$ is continuous for each $n$,
and that $\frac{1}{n} | \phi_{f,n}| \le C_f$ for some locally bounded function
$f \mapsto C_f \in \R$.
By the Subaddditive Ergodic Theorem, $\Phi_f = \lim_{n \to \infty} \frac{1}{n} \phi_{f,n}$
is defined $m$-almost everywhere.
Our result is: 

\begin{scho}
The points of continuity of the map $f \in \Diff_m^r(M) \mapsto \Phi_f \in L^1(m)$
form a residual subset.
\end{scho}

\begin{rem}\label{r.other measures}
Theorems \ref{t.CED}, \ref{t.PIS} and \ref{t.CLS} 
remain true (with identical proofs)
if Lebesgue measure $m$ is replaced by any other Borel probability $\mu$.
However, the spaces $\Diff_\mu^r(M)$ are in general very small,
and we couldn't conceive of any applications.
\end{rem}

\section{More Ingredients}\label{s.ingredients}

\subsection{Known \texorpdfstring{$C^1$}{C1}-Generic Results}\label{ss.c1 old stuff}

Here we collect some previously known residual properties for volume-preserving maps.
The first is the volume-preserving version of the Kupka--Smale Theorem, see \cite{Robinson}:

\begin{otherthm}\label{t.KS}
Assume $\dim M \ge 3$, $r \in \Z_+$.
Generically in $\Diff^r_m(M)$, every periodic orbit is hyperbolic,
and for every pair of periodic points $p$ and $q$, the manifolds $W^u(p)$ and $W^s(q)$ are transverse.
\end{otherthm}

The next is a ``connecting'' property:

\begin{otherthm}\label{t.connecting}
Assume $\dim M \ge 3$.
Generically in $\Diff^1_m(M)$,
if $p$ and $q$ are periodic points with $\dim W^u(p) \ge \dim W^u(q)$ then
$W^u(\cO(p)) \cap W^s(\cO(q))$ is dense in $M$.
\end{otherthm}

Indeed, Arnaud shows that generically
if $p$ and $q$ are periodic points with $\dim W^u(p) \ge \dim W^u(q)$ then
$$
W^u(\cO(p)) \cap W^s(\cO(q))
\quad \text{is dense in} \quad
\cl{W^u(\cO(p))} \cup \cl{W^s(\cO(q))}
$$
(see \cite{Ar_conn}, Proposition~18 and \S 1.5).
The latter set generically is the whole manifold $M$.
More precisely,
Bonatti and Crovisier had shown that \emph{each} homoclinic class\footnote{The 
\emph{homoclinic class} of a hyperbolic periodic point $p$
is the closure of the set of points of transverse intersection between 
$W^u(\cO(p))$ and $W^s(\cO(p))$.}
equals $M$ (see \cite{BC}, Theorem~1.3 and its proof on page~79;
here we use the assumption that $M$ is connected).
Hence Theorem~\ref{t.connecting} holds.
(It is also shown in \cite{BC} that the generic $f$ is transitive, but we won't use this.)

\begin{otherthm}[\cite{BV}] \label{t.BV} 
For a generic $f$ in $\Diff_m^1(M)$ and for $m$-a.e.\ $x\in M$,
the Oseledets splitting along the orbit of $x$ is (trivial or) dominated.
\end{otherthm}

\begin{corol}\label{c.BV}
For a generic $f$ in $\Diff_m^1(M)$, 
if   $G_i = \{ x \in M ; \; \lambda_i(f,x) > \lambda_{i+1}(f,x) \}$ has positive measure,
then there exist a nested sequence of measurable sets $\Lambda_1 \subset \Lambda_2 \subset \cdots \subset G_i$
such that ${m(G_i \setminus \Lambda_n)} \to 0$ as $n \to \infty$,
and each  each $\Lambda_n$ is $f$-invariant and has a dominated splitting of index $i$.
\end{corol}

The residual set of Corollary~\ref{c.BV} is the same as in Theorem~\ref{t.BV}.
(In fact, it is the set of points of continuity of all $m$-integrated Lyapunov exponents, see \cite{BV}.)



The following is the volume-preserving version of a result from \cite{ABC}
related to Ma\~{n}\'{e}'s Ergodic Closing Lemma:

\begin{otherthm}\label{t.ABC ECL}
For a generic $f$ in $\Diff_m^1(M)$, the following holds:
Given any $\mu \in \cM_\mathrm{erg}(f)$
there is a sequence of measures $\mu_n \in \cM_\mathrm{erg}(f)$, 
each supported on a periodic orbit,
such that:
\begin{itemize}
\item $\supp \mu_n$ converges to $\supp \mu$ in the Hausdorff topology; 

\item $\mu_n$ converges to $\mu$ in the weak-star topology;

\item the Lyapunov exponents of $f$ with respect to $\mu_n$
converge to the exponents with respect to $\mu$.
\end{itemize}
\end{otherthm}

\begin{proof}
The same statement for the dissipative case is Theorem~4.1 from~\cite{ABC},
and their proof applies to our volume-preserving situation, 
using Kupka--Smale Theorem~\ref{t.KS} and the (easier) volume-preserving version \cite{Ar_ECL}
of Ma\~{n}\'{e}'s Ergodic Closing Lemma \cite{Mane}.
\end{proof}

We will need an extension of the result above 
that deals with non-ergodic measures:

\begin{otherthm}\label{t.ECL plus}
For a generic $f$ in $\Diff_m^1(M)$, the following holds:
Given any $\mu \in \cM(f)$
there is a sequence of measures $\mu_n \in \cM(f)$, each with finite support,
such that:
\begin{itemize}
\item $\supp \mu_n$ converges to $\supp \mu$ in the Hausdorff topology; 

\item letting $\cL = \cL_f : M \to M \times \R^d$ be given by
\begin{equation}\label{e.trick}
\cL(x) = \big(x, \lambda_1(f,x), \ldots, \lambda_d(f,x) \big) , 
\end{equation}
then the sequence of measures $\cL_* \mu_n$ converges to $\cL_* \mu$
in the weak-star topology
(and in particular $\mu_n \to \mu$ as well).
\end{itemize}
\end{otherthm}

\begin{proof}
Let $f$ be generic in the sense of Theorem~\ref{t.ABC ECL}.
Thus the conclusion holds for ergodic measures,
and we will show that it also holds for any $\mu \in \cM(f)$.

Since $\bkappa_{f,\mu}$ is a decomposition of $\mu$,  
we have
$$
\cL_* \mu = \int \cL_* \nu \, d\bkappa_{f,\mu}(\nu) \, .
$$
We apply Lemma~\ref{l.combination} to approximate this integral by a finite convex combination.
Thus we find ergodic measures $\nu_1$, \ldots, $\nu_k$ with supports contained in $\supp \mu$ 
and positive numbers $c_1$, \ldots, $c_k$ with $\sum c_i = 1$
such that $\sum c_i \cL_* \nu_i$ is weak-star-close to $\cL_* \mu$.
By Theorem~\ref{t.ABC ECL}, for each $\nu_i$ we 
take $\tilde \nu_i \in \cM(f)$ supported on a finite set Hausdorff-close
to $\supp \nu_i$ such that $\cL_* \tilde \nu_i$ is weak-star-close to $\cL_* \nu_i$.
Thus the measure $\tilde \mu = \sum c_i \tilde \nu_i$
is supported on a finite set Hausdorff-close to $\supp \mu$,
and is such that $\cL_* \tilde \mu$ is weak-star-close to $\cL_* \mu$,
as desired. 
\end{proof}

\subsection{\texorpdfstring{$C^1$}{C1} Dominated Pesin Theory} \label{ss.dom pesin} 

We will need the fact that 
domination plus nonuniform hyperbolicity guarantees the existence of unstable and stable manifolds.
This has been claimed long ago by Ma\~n\'e~\cite{Mane_ICM} and 
recently made precise by Abdenur, Bonatti, and Crovisier \cite[\S 8]{ABC}.
However, their result does not fit directly to our needs,
and thus we take an independent approach.
More precisely, we first give a sufficient condition~\eqref{e.condition}
for the existence of a large stable manifold (Theorem~\ref{t.manifold}) at a given point,
and then we estimate the measure of the set of points that satisfy this condition, 
based on information about the Lyapunov exponents (Lemma~\ref{l.block}).

\subsubsection{Existence of Invariant Manifolds}\label{sss.manifolds}

Fixed $f \in \Diff^1(M)$, the (Pesin) stable set at a point $x\in M$ is
$$
W^s(x) = \left\{ y \in M; \; \limsup_{n \to +\infty} \frac{1}{n} \log d(f^n y, f^n x) < 0 \right\}
$$
and analogously for the unstable set.

From now on, fix $f \in \Diff^1(M)$.
Assume $\Lambda \subset M$ is an $f$-invariant Borel set with a dominated splitting
$T_\Lambda M = E^{cu}\oplus E^{cs}$.

For each $\ell \in \Z_+$, let 
$\block^s(\ell,f)$ be the set of points
$x \in \Lambda$ such that:
\begin{equation}\label{e.condition}
\frac{1}{n} \sum_{j=0}^{n-1} \log \|Df^\ell (f^{\ell j} x) | E^{cs} \| < -1
\quad \text{for every $n \in \Z_+$.}
\end{equation}
Also define $\block^u(\ell,f)$ as $\block^s(\ell, f^{-1})$, that is, the set of $x \in \Lambda$
such that $\frac{1}{n} \sum_{j=0}^{n-1} \log \|Df^{-\ell} (f^{-\ell j} x) | E^{cu} \| < -1$
for every $n \in \Z_+$.
The sets $\block^s(\ell,f)$ and $\block^u(\ell,f)$ are called 
\emph{unstable and stable Pesin blocks}.
We also denote $\block(\ell,f) = \block^s(\ell,f) \cap \block^u(\ell,f)$,
and call this set a \emph{Pesin block}.\footnote{Although this definition does not 
coincide with the usual one  in Pesin Theory, as e.g.\ \cite[\S 2.2.2]{BP_book},
we believe there is no risk of confusion.}


We fix cone fields $\cC^{cu}$, $\cC^{cs}$ 
around $E^{cu}$, $E^{cs}$ that are strictly invariant.
More precisely, for each $y\in \Lambda$ the open cone $\cC^{cu}_x \subset T_x M$
contains $E^{cu}_x$,  is transverse to $E^{cs}_x$, and the closure of its image by $Df(x)$ is contained in $E^{cs}_{fx}$;
analogously for $\cC^{cs}$. 
These cones can be extended to a small open neighborhood $V$ of $\cl \Lambda$,
so that strict invariance still holds for all points in $V$ that are mapped inside $V$.
If $g$ is sufficiently $C^1$-close to $f$ then the cone fields remain strictly invariant and
there is a dominated splitting over the maximal $g$-invariant set in $V$.

Let $x \in \Lambda$, $r>0$ be small,
and $\phi$ be a $C^1$ map from the ball of radius $r$ around $0$ in $E^{cs}(x)$ to $E^{cu}(x)$.
Let $D$ be the graph of the map $v \mapsto \exp_x(v + \phi(v))$.
If in addition the tangent space of $D$ at each point is contained in $\cC^{cs}$
and equals $E^{cs}(x)$ at $x$
then we say that $D$ is a \emph{center-stable disk} of radius $r$ around $x$.

\begin{otherthm}[Stable Manifold]\label{t.manifold} 
Consider an $f$-invariant set $\Lambda$ with a dominated splitting. 
For each $\ell \in \Z_+$ there exists $r>0$
such that if $x \in \block^s(\ell,f)$  
then $W^s(x)$ contains a center stable disk of radius $r$ around $x$.
Moreover, the same $r$ works for every diffeomorphism sufficiently
(depending on $\ell$) $C^1$-close to $f$.
\end{otherthm}

This result can be deduced from the Plaque Family Theorem
from \cite{HPS} (see also \cite{ABC}).
We prefer, however, to give a direct proof:

\begin{proof}
We work on exponential charts.  
Fix $\ell$ and take $x \in \block^s(\ell,f)$.
Let $c_n=\sum_{k=0}^{n-1} \log \|Df^\ell | E^{cs}(f^{k\ell}(x))\|$,
and let $B_n$ be the ball of radius $2 r e^{c_n} e^{n/2}$ around
$f^{n \ell}(x)$.  Let $D^n_n$ be the intersection of $B_n$ with
the affine space through $f^{n \ell}(x)$ tangent to
$E^{cs}(f^{n \ell}(x))$.  Define
$D^k_n$ for $k=n-1,\ldots,0$ by setting $D^k_n$ as the intersection of $B_k$
with $f^{-\ell}(D^{k+1}_n)$.  Notice that if $r$
is small then each $D^k_n$ will be tangent to the cone field, and in fact its
tangent space will be close to $E^{cs}(f^{k \ell}(x))$.  By the definition
of $\block^s(\ell,f)$, we see that
$\partial D^k_n \subset \partial B_k$ for each $0 \leq k \leq n$.  We claim
that the tangent space to $D^0_n$ is uniformly equicontinuous: for every
$\epsilon>0$ there exists $\delta>0$ such that
$d(T_y D^0_n,T_{y'} D^0_n)<\epsilon$ whenever $y$, $y' \in D^0_n$ are at
distance at most $\delta$.  Thus any accumulation  point of $D^0_n$ is 
a center stable disk $D^0$ of radius at least $r$
which is clearly contained in $W^s(x)$.

To see the claim, observe that domination implies that 
there are constants $C$, $\gamma>0$
such that for every $k$, if $y$, $y' \in D^0_n$ are sufficiently close (depending on $k$)
and 
$F$, $F'$ are subspaces tangent to $\cC^{cs} (f^{k\ell}(y))$ and  $\cC^{cs} (f^{k\ell}(y'))$ respectively,
then 
$$
d \big( (Df^{k\ell}(y))^{-1}(F), (Df^{k \ell}(y'))^{-1}(F') \big) < C e^{-\gamma k} \, .
$$ 
It follows that for each $0 \leq k \leq n$,
$$
d \big( T_y D^0_n,Df^{k \ell}(y)^{-1}(E^{cs}(f^{k\ell}(x))) \big)< C e^{-\gamma k} \, .
$$
Thus for every $\epsilon>0$, if $k \geq 0$ is minimal with 
$C e^{-\gamma k}<\epsilon/3$, and $d(y,y')$ is sufficiently small (depending on $k$),
then
$$
d(T_y D^0_n,T_{y'} D^0_n) < \frac{2\epsilon}{3} + 
d(Df^{k \ell}(y)^{-1}(E^{cs}(x)), Df^{k \ell}(y')^{-1}(E^{cs}(x)))<\epsilon,
$$
as claimed.
\end{proof}

\subsubsection{The Size of the Pesin Blocks}

In order to extract useful consequences from Theorem~\ref{t.manifold},
we need to estimate the measure of the Pesin blocks.
We will show that if $\lambda^{cs} = \lim_{\ell \to +\infty} \frac{1}{\ell} \log \|Df^{\ell} | E^{cs} \|$ is 
negative on most of $\Lambda$
then $\block^s(\ell,f)$ covers most of $\Lambda$, provided $\ell$ is large enough.
This follows from the next lemma,
which works for any $f$-invariant measure.
The lemma is also suitable to study the variation of the Pesin block with the diffeomorphism.

\begin{lemma} \label{l.block}
Let $\mu \in \cM(f)$.
Assume that $\eta>0$, $\alpha>0$, and $\ell \in \Z_+$ satisfy the following conditions:
\begin{gather}
\mu \{ x \in \Lambda ; \; \lambda^{cs}(x) > -\alpha \} < \eta ,  \label{e.c1}\\
\ell > \frac{1}{\alpha \eta} \, , \label{e.c2} \\
\int_{\Lambda} \left| \frac{1}{\ell} \log \|Df^{\ell} | E^{cs} \| - \lambda^{cs} \right| \, d\mu < \alpha \eta  \label{e.c3}.
\end{gather}
Then
$$
\mu \big( \Lambda \setminus \block^s(\ell,f) \big) < 3 \eta.
$$
\end{lemma}


To see how the lemma can be applied, assume, for example, that
$\lambda^{cs}<0$ $\mu$-almost everywhere on $\Lambda$.
Given a small $\eta>0$ we first take $\alpha$ satisfying \eqref{e.c1},
and then choose $\ell$ satisfying \eqref{e.c2} and \eqref{e.c3}.
The lemma then says that the Pesin block $\block^s(\ell,f)$ is large.

\begin{proof}
For $x \in \Lambda$, let
\begin{equation}\label{e.maximal function}
\phi(x)   =  \log \|Df^{\ell}(x) | E^{cs} \|, \quad
\phi^*(x) = \max_{n \ge 1} \frac{1}{n} \sum_{j=0}^{n-1} \phi(f^{\ell j}(x)) .
\end{equation}
Thus $\block^s(\ell,f) = \{ \phi^* < -1\}$.
Applying the Maximal Ergodic Theorem
to the restriction of the map $f^\ell$ to the (invariant) set of points $x \in \Lambda$ where $\lambda^{cs}(x) \le -\alpha$,
we obtain
$$
\int_{\{\phi^* \ge -1 \} \cap \{ \lambda^{cs} \le -\alpha\} } (\phi+1) \, d\mu \ge 0.
$$
Therefore
\begin{alignat*}{2}
0
&\le \int_{\{\phi^* \ge -1 \} \cap \{ \lambda^{cs} \le -\alpha\} } \frac{\phi + 1}{\ell} \, d\mu \\
&\le \alpha\eta + \int_{\{\phi^* \ge -1 \} \cap \{ \lambda^{cs} \le -\alpha\} } \frac{\phi}{\ell}\, d\mu 
&\quad&\text{(by \eqref{e.c2})} \\
&\le 2\alpha\eta +\int_{\{\phi^* \ge -1 \} \cap \{ \lambda^{cs} \le -\alpha\} } \lambda^{cs} \, d\mu 
&\quad&\text{(by \eqref{e.c3})} \\
&\le 2\alpha\eta -\alpha \mu \big(\{\phi^* \ge -1 \} \cap \{ \lambda^{cs} \le -\alpha\} \big)  \, ,
\end{alignat*}
and $\mu \big(\{\phi^* \ge -1 \} \cap \{ \lambda^{cs} \le -\alpha\} \big) \le 2\eta$.
It follows from \eqref{e.c1} that the set $\{\phi^* \ge -1 \}$ has $\mu$-measure less than $3\eta$, as we wanted to show.
\end{proof}

\subsection{\texorpdfstring{$C^2$}{C2} Pesin and Ergodicity}\label{ss.RRTU}

Since we will use Pesin Theory, the following result will have an important role:

\begin{otherthm}[\cite{Avila}] \label{t.Avila}
The subset $\Diff_m^2(M)$ of $\Diff_m^1(M)$ is dense.
\end{otherthm}

For the rest of this subsection, let $f$ be a fixed $C^2$ volume-preserving diffeomorphism.
By Pesin Theory\footnote{A recent comprehensive reference in book form is \cite{BP_book}.}, 
$W^u(x)$ and $W^s(x)$ (as defined in \S\ref{sss.manifolds})
are immersed manifolds for every $x$ in 
a full probability Borel set $R_f$.
The dimension of $W^u(x)$ is the number (with multiplicity) of positive Lyapunov exponents at $x$,
and symmetrically for $W^s(x)$.

Following \cite{RRTU},
we define the \emph{unstable Pesin heteroclinic class} 
of a hyperbolic periodic point $p$ as 
\begin{multline*}
\class^u(p) = \big\{ x \in R_f  ; \;  W^u(x) \text{ intersects transversely} \\
W^s(\cO(p)) \text{ in at least one point} \big\}.
\end{multline*}
This is always an invariant Lebesgue measurable set.\footnote{Here is a proof of measurability: 
For any $y \in M$, let $U_y \subset T_y M$ be the set of vectors 
that are exponentially contracted under negative iterations;
this is a Borel measurable function.
Notice that if $y$ belongs to a Pesin manifold $W^s(x)$ then $T_y W^u(x) = U_y$.
Let $Y$ be the subset of $y \in W^s(\cO(p))$ such that $U_y$ is transverse to $W^s(\cO(p))$; this is a Borel set.
Let $Z$ be the subset of $\cP \times Y$ formed by pairs $(x,y)$ such that $y \in W^u(x)$;
this is a Borel set.
By the Measurable Projection Theorem \cite[Theorem~III.23]{CV},
the projection $\class^u(p)$ of $Z$ in the first coordinate is Lebesgue measurable.}
This set has the following $u$-saturation property:
for $m$-almost every $x$ in $\class^u(p)$, almost every point in $W^u(x)$
(with respect to Riemannian volume on the submanifold)
belongs to $R_f$ and thus to $\class^u(p)$.
This follows from the absolute continuity of Pesin manifolds, see \cite[\S 8.6.2]{BP_book}

Analogously we define the \emph{stable Pesin heteroclinic class} $\class^s(p)$. 
The \emph{Pesin heteroclinic class}\footnote{This set is called an \emph{ergodic homoclinic class} in \cite{RRTU}.} 
of $p$ is defined as $\class(p) = \class^u(p)\cap \class^s(p)$. 

The usefulness of Pesin heteroclinic classes comes from the following result:

\begin{otherthm}[Criterion for Ergodicity; Theorem~A from \cite{RRTU}] \label{t.RRTU}
Let $p$ be a hyperbolic periodic point for $f \in \Diff_m^2(M)$.
If both sets $\class^u(p)$ and $\class^s(p)$ 
have positive $m$-measure 
then they are equal $m$-mod~$0$,
and the restriction of $m$ to any of them is an ergodic measure for $f$.
\end{otherthm}

%
%


Let us observe two properties of Pesin heteroclinic classes:

\begin{lemma}[Remark 4.4 from \cite{RRTU}]\label{l.lambda}
If $p$ and $q$ are hyperbolic periodic points such that $W^u(\cO(p))$ and $W^s(\cO(q))$
have nonempty transverse intersection then 
$\class^u(p) \subset \class^u(q)$ and $\class^s(q) \subset \class^s(p)$. 
\end{lemma}

\begin{lemma}\label{l.ess cl}
If $p$ is a hyperbolic periodic point with $m(\class^u(p))>0$ then
$W^u(\cO(p)) \subset \ess \cl \class^u(p)$.
\end{lemma}

Here $\ess \cl X$ denotes the \emph{essential closure} of a set $X\subset M$, that is, 
the set of points $x\in M$ such that $m(X\cap V)>0$ for every neighborhood $V$ of~$x$.

\begin{proof}[Proof of Lemma~\ref{l.ess cl}]
Assume that $m(\class^u(p))>0$.
Recall that the set $R_f$ is the union of a sequence of 
blocks, in each of these there are local Pesin manifolds of uniform size
that depend on the point in a uniformly continuous way with respect to the $C^1$ topology.
Therefore
we can find a continuous family of disks $D_y$, where $y$ runs over a compact subset $K$ of $W^s(\cO(p))$,
with the following properties: 
each disk $D_y$ contains $y$, is contained in a Pesin stable manifold, and is transverse to $W^s(\cO(p))$;
the union $\bigcup_{y \in K} D_y$ has positive measure.
Now let $U$ be any open set intersecting $W^u(\cO(p))$.
By the Lambda Lemma, there is $n>0$ such that $f^{-n}(U)$ intersects all disks $D_y$, $y\in K$.
By the absolute continuity of Pesin manifolds,
this implies that 
$\bigcup_{y} D_y \cap f^{-n}(U)$ has positive $m$ measure
(use \cite[Corollary~8.6.9]{BP_book}).
Since the class $\class^u(p)$ contains mod~$0$ the union of disks 
and is invariant, we conclude that its intersection with $U$ has positive measure.
\end{proof}

\section{Proof of the Main Result} \label{s.proof}

In this section we use all previous material to prove Theorem~\ref{t.main}.
We assume from now on that $\dim M \ge 3$, because otherwise the theorem is reduced to the
Ma\~n\'e--Bochi Theorem \cite{B_ETDS}.

Let $\cR \subset \Diff_m^1(M)$ be the intersection of the residual sets 
given by Theorems~\ref{t.CED}, \ref{t.PIS}, \ref{t.CLS}, and \ref{t.KS} with $r=1$, 
and also Theorems \ref{t.connecting}, \ref{t.BV}, and \ref{t.ECL plus}.
Fix any $f \in \cR$;
we will show that it satisfies the conclusions of Theorem~\ref{t.main}.
This will be done in two steps:
\begin{itemize}
\item In Lemma~\ref{l.main} we show that $C^2$ perturbations of $f$
have an ergodic component with positive Lebesgue measure (and some additional properties).
\item Using continuity of the ergodic decomposition at the original $C^1$-diffeomorphism $f$ 
(along with other things),
we show that it already has the desired properties.
\end{itemize}

For $i \in \{1,\ldots, d-1 \}$, let $\NUH_i(f)$ be the set of points $x\in \NUH(f)$ that have index $i$,
that is, the set of Lyapunov regular points such that $\lambda_i(f,x)>0>\lambda_{i+1}(f,x)$.

\begin{lemma}\label{l.main}
Let $f \in \cR$ and $i\in \{1,\ldots, d-1\}$. 
Assume that $\Lambda \subset \NUH_i(f)$ is a Borel $f$-invariant set of positive measure 
that has a dominated splitting of index $i$.
Then for any $\eps>0$,
there exist finitely many (hyperbolic) periodic points $p_1$, \ldots, $p_J$ of~$f$ of index $i$
with the following properties:
For every volume-preserving $C^2$-diffeomorphism $g$ sufficiently $C^1$-close to~$f$,
there exist $j \in \{1,\ldots,J\}$ such that
if $p^g = p_j^g$ denotes the continuation of~$p_j$ (that is, the unique $g$-periodic point that
is close to $p_j$ and has the same period),
then:
\begin{enumerate}
\item\label{i.erg class} 
the measure $m | \class(p^g,g)$ is non-zero and ergodic for $g$;
\item\label{i.NUH class} 
$\class(p^g, g) \subset \NUH_i(g)$ mod~$0$;
\item\label{i.neighbor} 
$\class(p^g, g) \subset B_\eps(\Lambda)$ mod~$0$;
\item\label{i.sym dif} 
$m\big( \class(p^g,g) \vartriangle \Lambda \big) < \eps$.
\end{enumerate}
\end{lemma}

Notice that if $m(\NUH_i(f))>0$ then by Corollary~\ref{c.BV}
it always exists a set $\Lambda$ satisfying the
hypotheses of the lemma; moreover $\Lambda$ can be taken so that
the measure of $\NUH_i(f) \setminus \Lambda$ is as small as desired. 
(In fact, since we will prove later that $m|\NUH_i(f)$ is ergodic,
any set $\Lambda$ that satisfies the hypotheses of the lemma coincides mod~$0$ with $\NUH_i(f)$.)

\begin{proof}[Proof of Lemma~\ref{l.main}]
Fix a Borel invariant set $\Lambda \subset \NUH_i(f)$ with a dominated splitting 
$E^{cu}\oplus E^{cs}$ of index $i$.
Also fix a positive number $\eps$, which we can assume less than $m(\Lambda)$.
As in \S\ref{sss.manifolds}, we fix a neighborhood $V=B_{r_*}(\Lambda)$
and strictly invariant cone fields $\cC^{cu}$, $\cC^{cs}$ on it.
Then, for every $g$ sufficiently $C^1$-close to $f$ and $\ell \in \Z_+$ 
we let $\block^s(\ell,g)$ and $\block^u(\ell,g)$ be the 
associated Pesin $s$- and $u$-blocks, 
viewed as subsets of the maximal $g$-invariant set contained in
$B_{r_*}(\Lambda)$. 

Let $\eta = \eps/200$.
Since
$\lambda_{i+1}(f,x) < 0 < \lambda_i(f,x)$ for $x \in \Lambda$,
we can find $\alpha \in (0,1)$ 
such that 
\begin{align}
m \big\{x \in \Lambda ; \; \lambda_{i+1}(f,x) > - \alpha \text{ or } \lambda_{i}(f,x) < \alpha \big\} &< \eta 
\quad \text{and} \label{e.alpha1} \\
m \big\{x \in \Lambda ; \; \lambda_{i+1}(f,x) = - \alpha \text{ or } \lambda_{i}(f,x) = \alpha \big\} &= 0 \, .
\label{e.alpha0}
\end{align}
Let $\ell > 1 /(\alpha \eta)$ be such that 
\begin{multline}\label{e.l}
\int_{\Lambda} \Big| \frac{1}{\ell} \log \|Df^{\ell} | E^{cs}(x) \| - \lambda_{i+1}(f,x) \Big| \, dm(x)
\\ + 
\int_{\Lambda} \Big| \frac{1}{\ell} \log \|Df^{-\ell} | E^{cu}(x) \| - \lambda_{i}(f,x) \Big| \, dm(x)
< \alpha \eta \, .
\end{multline}
Also fix a positive $r<r_*$
such that if $g$ is close to $f$ and $x$, $y$ are points in $\block(\ell,g)$
whose distance is less than $r$ then the Pesin manifolds $W^u(x)$ and $W^s(y)$
have a transverse intersection.

Once these constants are fixed, let us prove three sublemmas.

\begin{subl}[The Pesin block is robustly large]\label{sl.robust block}
If $g$ is sufficiently close to $f$ then
$$
m \big( \Lambda \setminus \block(\ell,g) \big) < 61 \eta.
$$
\end{subl}

\begin{proof}
Let $C = \max_M |\log \|Df^{\pm 1}\||$.
By Theorem~\ref{t.PIS} (and the fact that $f\in \cR$),
for any $g$ sufficiently close to $f$ there exists a $g$-invariant Borel set $\Lambda^g \subset B_{r^*}(\Lambda)$
such that $m(\Lambda^g \vartriangle \Lambda) < C^{-1}\alpha\eta$.
Taking $g$ sufficiently close to $f$, we can guarantee that 
$$
\left| \frac{1}{\ell} \log \|Dg^{\ell} | E^{\star}_g \| - \frac{1}{\ell} \log \|Df^{\ell} | E^{\star}_f \| \right| 
< \alpha\eta
\quad \text{on $\Lambda \cap \Lambda^g$, $\star = cu$, $cs$.}
$$
By Theorem~\ref{t.CLS} (and the fact that $f\in \cR$) 
we can also suppose that for $j=i$, $i+1$, 
the $L^1$-distance between $\lambda_{j}(g, \mathord{\cdot})$ and
$\lambda_{j}(f, \mathord{\cdot})$ is less than $\alpha\eta$,
and small enough so that
$$
m \big\{x \in M ; \; |\lambda_{j}(g,x) - \lambda_{j}(f,x)| > \alpha/2 \big\} < \eta \, .  
$$

We will check that the hypotheses of Lemma~\ref{l.block}
are satisfied with $\mu=m$, $g$ in the place of $f$, $\Lambda^g$ in the place of $\Lambda$,
$\alpha/2$ in the place of $\alpha$, and $10\eta$ in the place of $\eta$.
That is, we have 
\begin{gather}
m \{ x \in \Lambda^g ; \; \lambda_{i+1}(g,x) > -\alpha/2 \} < 10\eta ,  \label{e.c1 again}\\
\ell > 1/(5\alpha \eta) \, , \label{e.c2 again} \\
\int_{\Lambda^g} \left| \frac{1}{\ell} \log \|Dg^{\ell} | E^{cs}_g \| - \lambda_{i+1}(g) \right| \, dm < 5 \alpha \eta  \label{e.c3 again}.
\end{gather}

First, the set $\{x \in \Lambda^g ; \; \lambda_{i+1}(g,x) > -\alpha/2 \}$ is contained in the union
of the three sets
$$
\Lambda^g \setminus \Lambda, \ 
\{x \in \Lambda ; \; \lambda_{i+1}(f,x) > -\alpha \}, \   
\{x \in M ; \; |\lambda_{i+1}(g,x) - \lambda_{i+1}(f,x)| > \alpha/2\} \, ,
$$
and each of them has measure less than $\eta$;
thus \eqref{e.c1 again} holds.
Second, \eqref{e.c2 again} is true by definition of $\ell$.
Third, 
\begin{align*}
\int_{\Lambda^g} &\left| \frac{1}{\ell} \log \|Df^{\ell} | E^{cs}_g \| - \lambda_{i+1}(g) \right| \\
&\le 
\int_{\Lambda^g \cap \Lambda} \left| \frac{1}{\ell} \log \|Dg^{\ell} | E^{cs}_g \| - \lambda_{i+1}(g) \right| 
+ C m(\Lambda^g \setminus \Lambda) \\
&\le
\alpha\eta + 
\int_{\Lambda^g \cap \Lambda} \left| \frac{1}{\ell} \log \|Df^{\ell} | E^{cs}_f \| - \lambda_{i+1}(f) \right| 
+ \|\lambda_{i+1}(f) - \lambda_{i+1}(g)\|_{L^1} +  \alpha \eta \\
&\le 4\alpha\eta \, .
\end{align*}
So \eqref{e.c3 again} is also satisfied. 

Lemma~\ref{l.block} then gives
$m(\Lambda^g \setminus \block^s(\ell,g)) < 30\eta$.
An analogous estimate gives $m(\Lambda^g \setminus \block^u(\ell,g)) < 30\eta$.
It follows that
\[
m(\Lambda \setminus \block(\ell,g)) \le m(\Lambda^g \setminus \block(\ell,g)) 
+ m(\Lambda \setminus \Lambda^g) < 61 \eta. \qedhere
\]
\end{proof}

Let $m_\Lambda$ be the $f$-invariant measure $m_\Lambda(A) = m(\Lambda\cap A)/m(\Lambda)$.

\begin{subl}\label{sl.comparasion}
If $\mu$ is a probability measure sufficiently weak-star close to $m_\Lambda$
then
$$
m_\Lambda \big( B_r(G) \big) \ge \mu(G) - \eta
\quad \text{for any Borel set $G$.}
$$
\end{subl}

\begin{proof}
We choose an $r$-fine partition of unity, that is, a family of
continuous non-negative functions $\psi_j: M \to \R$, $j=1,\ldots,J$
such that $\sum_j \psi_j = 1$ and
each set
$\supp \psi_j = \cl{\{\psi_j \neq 0\}}$
has diameter less than~$r$.
Now assume that $\mu$ is a measure
close enough to $m_\Lambda$ so that
$$
\int \psi_j \, d\mu \le
\int \psi_j \, dm_\Lambda + \frac{\eta}{J} \quad \text{for each $j$.}
$$
Given a Borel set $G$,
consider all functions $\psi_j$ such that $\supp \psi_j \subset B_r(x)$ for some $x \in G$,
and let $\hat \psi$ be their sum.
Notice that
$$
\charac_G \le \hat\psi \le  \charac_{B_r(G)} \, .
$$
Therefore
\[
\mu(G)
\le \int \hat\psi \, d\mu
\le \int \hat\psi \, d\mu + \eta
\le m_\Lambda \big( B_{r}(G) \big) + \eta. \qedhere
\]
\end{proof}

\begin{subl}[Covering most of $\Lambda$ by balls around good periodic points]\label{sl.good periodic}
There exists a finite $f$-invariant set $F \subset B_r(\Lambda)$
such that
$$
m \Big( \Lambda \setminus B_r\big(F \cap \block(\ell,f) \big) \Big) < 7 \eta \, .
$$
\end{subl}

\begin{proof}
The idea is to use Lemma~\ref{l.block} again.

By Theorem~\ref{t.ECL plus},
we can find a measure $\mu$ supported on a finite set $F$ 
that is Hausdorff-close to $\supp m_\Lambda = \ess \cl{\Lambda}$
(and in particular contained in $B_r(\Lambda)$)
such that $\cL_*\mu$ is weak-star close to $\cL_* m_\Lambda$,
where $\cL$ is given by \eqref{e.trick}.
In particular, $(\lambda_{i+1})_* \mu$ is close to $(\lambda_{i+1})_* m_\Lambda$.
It then follows from \eqref{e.alpha1} and \eqref{e.alpha0} that
$$
\mu \big\{x \in M ; \; \lambda_{i+1}(f,x) > - \alpha \big\} < \eta \, ,
$$
and in particular, condition \eqref{e.c1} holds (with $F$ in the place of $\Lambda$).

The proximity between $\cL_* \mu$ and $\cL_* m_\Lambda$
also implies that
the integrals of the function $\left| \frac{1}{\ell} \log \|Df^{\ell} | E^{cs} \| - \lambda_{i+1}(f) \right|$
with respect to the measures $\mu$ and $m_\Lambda$ are close.
(Indeed we can write the integral with respect to $m_\Lambda$ as
$$
\int_{M\times \R^d} \left| \frac{1}{\ell} \log \|Df^{\ell} | E^{cs}(x) \| - y_{i+1} \right| \, d (\cL_* m_\Lambda)(x,y_1,\ldots,y_d)
$$
and the integrand 
is a continuous function.)
In particular, condition \eqref{e.c3} (with $F$ in the place of $\Lambda$)
follows from \eqref{e.l}.

Thus we can apply Lemma~\ref{l.block}  
and get that $\mu(F \setminus \block^s(\ell,f)) < 3\eta$.
The same estimate holds for $\block^u(\ell,f)$.
Now applying Sublemma~\ref{sl.comparasion} to the set $G = F \cap \block(\ell,f)$
we obtain $m_\Lambda(B_r(G)) \ge \mu(G) - \eta > 1- 7\eta$, and in particular
\[
m(\Lambda \setminus B_r(G)) \le m_\Lambda(\Lambda \setminus B_r(G)) < 7 \eta \, . \qedhere
\]
\end{proof}

We continue with the proof of Lemma~\ref{l.main}.
Let $F$ be given by Sublemma~\ref{sl.good periodic}.
For each $p \in F$ and $g$ close to $f$, let $p^g$ denote the continuation of $p$,
and let $F^g = \{ p^g ;\; p \in F \}$.
Notice that 
$$
p \in F \cap \block(\ell,f) \ \Rightarrow \ 
p^g \in F^g \cap \block(\ell,g) \text{ for all $g$ sufficiently close to $f$} \, .
$$
Indeed, for periodic points, belongingness to the Pesin block involves only a finite number of (open) conditions.

Thus it follows from Sublemma~\ref{sl.good periodic} that
for $g$ sufficiently close to $f$,
$$
m \Big( \Lambda \setminus B_r\big(F^g \cap \block(\ell,g) \big) \Big) < 10 \eta \, .
$$
This, together with Sublemma~\ref{sl.robust block}, gives
\begin{equation}\label{e.robot rock}
m \bigg( \Lambda \setminus   \Big( B_r\big(F^g \cap \block(\ell,g) \big) \cap \block(\ell,g) \Big) \bigg) 
< 100 \eta = \frac{\eps}{2} < m(\Lambda) \, .
\end{equation}
In particular, there exist at least one 
point $p^g \in F^g \cap \block(\ell,g)$ such that 
$B_r(p^g) \cap \block(\ell,g)$ has positive measure.
It follows from the definition of $r$ that
$$
B_r(p^g) \cap \block(\ell,g) \subset \class(p^g,g) \bmod{0}.
$$
and so $m(\class(p^g,g))>0$.
Now assume that $g$ is $C^2$. 
Then, by Theorem~\ref{t.RRTU}, 
the restriction of $m$ to $\class(p^g,g)$ is an ergodic measure for $g$;
this proves part~(\ref{i.erg class}) of the lemma.
This measure gives positive weight to $\block(\ell,g)$, which is contained
in the $g$-invariant set $\NUH_i(g)$.
Ergodicity implies that $\NUH_i(g) \supset \class(p^g,g)$ mod~$0$,
which is part~(\ref{i.NUH class}) of the lemma.

We claim that this class $\class(p^g,g)$ does not depend on the choice of the point $p$.
More precisely, if $q$ is another point in $F \cap \block(\ell,f)$ such that 
$B_r(q^g) \cap \block(\ell,g)$ also has positive measure, then  $\class(p^g,g) = \class(q^g,g)$ mod~$0$.
Indeed, since $p$ and $q$ have the same index $i$,
the manifolds $W^u(\cO_f(p))$ and $W^s(\cO_f(q))$ have nonempty intersection by Theorem~\ref{t.connecting},
which is transverse by Theorem~\ref{t.KS}.
Assuming that $g$ is sufficiently close to $f$,
the unstable  manifolds of $\cO_g(p^g)$ still has 
a nonempty transverse intersection
with the stable manifold of $\cO(q^g)$.
Thus, by Lemma~\ref{l.lambda},
$\class^u(p^g, g) \subset \class^u(q^g, g)$ and $\class^s(q^g, g) \subset \class^s(p^g, g)$. 
Since those sets have positive measure, Theorem~\ref{t.RRTU}
implies that they are all equal mod~$0$.

It follows from the claim and \eqref{e.robot rock} that 
\begin{equation}\label{e.lovely to see you}
m\big( \Lambda \setminus \class(p^g,g) \big) < 100\eta < \eps/2.
\end{equation}

To complete the proof,
assume that $g$ is sufficiently close to $f$ so that, by Theorem~\ref{t.PIS},
it has an invariant set $\Lambda^g$ with
$$
\Lambda^g \subset B_\eps (\Lambda) \quad \text{and} \quad 
m \big(\Lambda^g \vartriangle \Lambda \big) < \eps/2 \, .
$$
Then
\begin{align*}
m\big( \Lambda^g \cap \class(p^g,g) \big) &\ge
m(\Lambda) - m \big(\Lambda^g \setminus \Lambda \big) - m\big( \Lambda \setminus \class(p^g,g) \big) \\
&> m(\Lambda) - \frac{\eps}{2} - \frac{\eps}{2}
>0.
\end{align*}
So ergodicity implies that 
$\class(p^g,g) \subset \Lambda^g$ mod~$0$.
In particular, $\class(p^g,g) \subset B_\eps(\Lambda)$ mod~$0$,
which is part~(\ref{i.neighbor}),
and  $m\big(\class(p^g,g) \setminus \Lambda\big) < \eps/2$,
which together with \eqref{e.lovely to see you}
gives part~(\ref{i.sym dif}).
The proof of Lemma~\ref{l.main} is completed.
\end{proof}  

\begin{proof}[Proof of Theorem~\ref{t.main}]
Take a diffeomorphism $f$ in the set $\cR$ described before.
If the set $\NUH(f)$ has zero measure then there is nothing to show, so assume this is not the case.
Take $i \in \{1, \ldots, d-1\}$ such that $m(\NUH_i(f))>0$.

\smallskip

\emph{Proof that $m|\NUH_i(f)$ is ergodic:}
Let $a = m(\NUH_i(f))$ and  $\mu = a^{-1} \cdot m|\NUH_i(f)$.
By contradiction, assume that $\mu$ is not ergodic for $f$.
Then, in the notation of Section~\ref{s.new}, we have $\bkappa_{f}(\{\mu\}) = 0$.
Let $\cU \subset \cM(f)$ be an open set containing $\mu$ 
with $\bkappa_{f}(\cU) < a$ and $\bkappa_{f} (\partial \cU) = 0$.
Using Theorem~\ref{t.Avila},
choose a sequence $g_n$ of $C^2$ volume-preserving diffeomorphisms converging to $f$ in the $C^1$-topology.
Using Lemma~\ref{l.main}, we can find for each sufficiently large $n$
a Borel set $H_n$  such that 
the measure
$m|H_n$ is non-zero, invariant and ergodic with respect to $g_n$,
and moreover $m \big( H_n\vartriangle \NUH_i(f) \big) \to 0$ as $n \to \infty$.
Denote by $\mu_n$  the normalization of $m|H_n$; 
then $\mu_n \to \mu$.
Since $\mu_n$ is $g_n$-ergodic, we have $\bkappa_{g_n}(\{\mu_n\}) = m(H_n) \to a$.
On the other hand, for sufficiently large $n$ we have  $\bkappa_{g_n}(\{\mu_n\}) \le \bkappa_{g_n}(\cU)$.
But, by Theorem~\ref{t.CED}, $\bkappa_{g_n}(\cU) \to \bkappa_{f}(\cU) < a$.
This contradiction proves ergodicity.

%

\smallskip

\emph{Proof that $\NUH_i(f)$ is essentially dense:}
By contradiction, assume this is not the case, thus 
there exists $z\in M$ and $\eps>0$ such that
$m(B_{2\eps}(z) \cap \NUH_i(f)) = 0$.
Let $\Lambda$ be the set of Lebesgue density points of $\NUH_i(f)$;
then 
\begin{equation} \label{e.hole}
\Lambda \cap B_{2\eps}(z) = \emptyset .	
\end{equation}
Since $f|\Lambda$ has a dominated splitting,
we can apply Lemma~\ref{l.main} and find periodic points $p_1$, \dots, $p_J$.
By Theorem~\ref{t.connecting}, each manifold $W^s(\cO(p_j))$ is dense in $M$.
Thus, $W^s(\cO(p_j^g),g)\cap B_\eps(z) \neq \emptyset$ for every $g$ sufficiently close to $f$ and every $j$.
Take a $C^2$ diffeomorphism $g$ very close to $f$;
then 
$W^s(\cO(p_j^g),g)\cap B_\eps(z) \neq \emptyset$ for every $j$.
Moreover, by  Lemma~\ref{l.main} there is $j$ such that
$\class(p_j^g)$ has positive measure.
By Lemma~\ref{l.ess cl}, the essential closure of $\class^s(p_j^g,g)$ 
(which equals $\ess \cl \class(p_j^g,g)$ by Theorem~\ref{t.RRTU})
contains $W^s(\cO(p_j),g)$;
in particular, $\class(p_j^g) \cap B_\eps(z)$ has positive measure.
Lemma~\ref{l.main} also says that 
$\class(p_j^g) \subset B_\eps(\Lambda)$ mod~$0$,
which contradicts \eqref{e.hole}.

%

\smallskip

\emph{Proof of the uniqueness of the index $i$:}
Let $k \in \{1, \ldots, d-1\}$ be such that $\NUH_k(f)$ has positive measure.
By symmetry, we can assume that $i \ge k$.
Applying Lemma~\ref{l.main} twice,
namely, to the sets of Lebesgue density points of $\NUH_i(f)$ and $\NUH_k(f)$,
we obtain periodic points
$p_1$, \ldots, $p_J$ of index $i$,
and $q_1$, \ldots, $q_L$ of index $k$.
By Theorem~\ref{t.connecting},
the manifolds $W^u(\cO_f(p_j))$ and $W^s(\cO_f(q_\ell))$ have nonempty intersection,
which is transverse by Theorem~\ref{t.KS}.
Now consider a $C^2$ diffeomorphism $g$ that is $C^1$-close to~$f$.
Then the manifolds $W^u(\cO_g(p_j^g))$ and $W^s(\cO_g(q_\ell^g))$ still intersect transversely.
Thus, by Lemma~\ref{l.lambda},
$\class^u(p^g_j, g) \subset \class^u(q^g_\ell, g)$ and $\class^s(q^g_\ell, g) \subset \class^s(p^g_j, g)$
for each $j$, $\ell$.
On the other hand, by Lemma~\ref{l.main},
there are $j$ and $\ell$ such that
$\class(p^g_j)$ has positive measure and is contain mod~$0$ in $\NUH_i(g)$,
and $\class(q^g_\ell)$ has positive measure and is contain mod~$0$ in $\NUH_k(g)$.
By Theorem~\ref{t.RRTU}, $\class(p^g_j) = \class(q^g_\ell)$ mod~$0$.
So $\NUH_i(g) \cap \NUH_k(g)$ has positive measure and therefore $k=i$,
as we wanted to show.

This completes the proof of Theorem~\ref{t.main}.
\end{proof}


\vfill

{\footnotesize


\noindent CNRS UMR 7586, Institut de Math\'ematiques de Jussieu,
Universit\'e de Paris~VII. France.

\noindent \textit{Current address:} IMPA, Estr.\ D.\ Castorina, 110, Rio de Janeiro, RJ 22460-320, Brazil.

\noindent \href{http://www.impa.br/%7Eavila}{\tt www.impa.br/$\sim$avila}

\noindent \href{mailto:artur@math.sunysb.edu}{\tt artur@math.sunysb.edu}

\bigskip


\noindent Departamento de Matem\'atica, Pontif\'{\i}cia Universidade Cat\'olica do Rio de Janeiro.

\noindent Rua Mq.\ S.\ Vicente, 225, Rio de Janeiro, RJ 22453-900, Brazil.

\noindent \href{http://www.mat.puc-rio.br/%7Ejairo}{\tt www.mat.puc-rio.br/$\sim$jairo}

\noindent \href{mailto:jairo@mat.puc-rio.br}{\tt jairo@mat.puc-rio.br}

}
\end{document}